\documentclass{amsart}
\usepackage{amsfonts}
\usepackage{amssymb, amsmath, amsthm}
\usepackage{amscd}
\usepackage{verbatim}
\usepackage{enumerate}
\usepackage{color}
\usepackage{tikz}
\usepackage{colordvi}
\usepackage{graphics}
\usepackage{graphicx}

\allowdisplaybreaks
\newtheorem{theorem}{Theorem}[section]

\newtheorem{alg}[theorem]{Algorithm}

\newtheorem{lemma}[theorem]{Lemma}

\newtheorem{remark}[theorem]{Remark}
\newtheorem{assumption}[theorem]{Assumption}

\numberwithin{equation}{section}

\newcommand{\beq}{\begin{equation}}
\newcommand{\eeq}{\end{equation}}
\newcommand{\beqa}{\begin{eqnarray}}
\newcommand{\eeqa}{\end{eqnarray}}

\newcommand{\nab}{\nabla}

\newcommand{\nr}[1]{\ensuremath{\left\|{#1} \right\|}}
\newcommand{\te}{\tilde e}
\newcommand{\grad}{\nabla}

\newcommand{\tu}{\tilde u}

\begin{document}

\title[Choice of optimization norm for AA optimization]{On the choice of optimization norm for Anderson acceleration of the Picard iteration for Navier-Stokes equations}
\author{Elizabeth Hawkins$^{*}$ and    Leo G Rebholz$^{*}$}
\dedicatory{\vspace{-10pt}\normalsize{ $^{*}$School of Mathematical and Statistical Sciences, Clemson University,
Clemson, SC 29634, USA\\
}}


\begin{abstract}

While the most recent Anderson acceleration (AA) convergence theory  [Pollock et al, {\it IMA Num. An.}, 2021] requires that the AA optimization norm  match the Hilbert space norm associated with the fixed point operator, in implementations the $\ell^2$ norm is perhaps the most common choice. Unfortunately, so far there is little research done regarding this discrepancy which might reveal when it is fine to use $\ell^2$.
To address this issue, we consider AA applied to the  Picard iteration for the Navier-Stokes equations (NSE) with varying choices of the AA optimization norm.  We first prove a sharpened and generalized convergence estimate for depth $m$ AA-Picard for the NSE with the $H^1_0$ AA optimization norm by using a problem-specific analysis, utilizing a sharper treatment of the nonlinear terms than previous AA-Picard convergence studies, removing a small data assumption, and developing new AA term identities in the NSE nonlinear term estimates.  Next, we prove 
a convergence result for when $L^2$ is used as the AA optimization norm, and this estimate is found to be very similar to that of the $H^1_0$ case.
While no analogous theory seems possible for the $\ell^2$ norm, several numerical tests were run to compare AA-Picard convergence with varying choices of AA optimization norm.  These tests revealed that convergence behavior was always similar for $L^2$ and $H^1_0$ and {\it usually but not always} similar for $\ell^2$: on a test problem for channel flow past a cylinder with coarser meshes, convergence of AA-Picard using $\ell^2$ performs significantly worse than using $L^2$ and $H^1_0$.  
\end{abstract}

\maketitle

\vspace{-.25in}
\section{Introduction}\label{s:intro}

	We consider an Anderson accelerated nonlinear solver for the Navier-Stokes equations (NSE) that model incompressible fluid flow, which are given for steady flows on a domain $\Omega \subset \mathbb{R}^d, d=2,3$ by
	
	\begin{equation}\label{NS}
		\left\{\begin{aligned}
			-\nu \Delta u+u\cdot\nabla u+ \nabla p&={f} \quad \text{in}~\Omega,\\
			\nabla\cdot {u}&=0\quad \text{in}~\Omega,\\
			{u}&=0 \quad \text{on}~~\partial\Omega,
		\end{aligned}\right.
	\end{equation}
where $u$ and $p$ are the unknown fluid velocity and pressure, $\nu>0$ is the kinematic viscosity,  and ${f}$ is a given external forcing (such as gravity or bouyancy). The Reynolds number $Re \sim \frac{1}{\nu}$ is a physical constant that describes the complexity of a flow: higher $Re$ is typically associated with more complex physics and non-unique solutions of \eqref{NS}.  We consider the system \eqref{NS} equipped with homogenous Dirichlet boundary conditions, but our results can extend to nonhomogeneous mixed Dirichlet/Neumann boundary conditions as well as to solving the time dependent NSE at a fixed time step in a time stepping scheme.  It is well known that \eqref{NS} admits weak solutions for any $\nu>0$ and $f\in H^{-1}(\Omega)$ \cite{Laytonbook}, and under a small data (sufficient) condition $\kappa:=M\nu^{-2} \|f \|_{H^{-1}}<1$ (where $M$ is a domain-size dependent constant, see section 2), solutions are unique \cite{GR86,Laytonbook,temam}.

A common nonlinear solver for \eqref{NS}  is  the Picard iteration, which is given by 
	\begin{align*}
		-\nu \Delta u_{k+1}+u_k\cdot\nabla u_{k+1}+ \nabla p_{k+1}&={f}, \\
		\nabla\cdot {u}_{k+1}&=0, \\
		u_{k+1} |_{\partial\Omega} &= 0.
	\end{align*}
The Picard iteration for the NSE is globally stable, and is globally linearly convergent with rate $\kappa$ if $\kappa<1$ \cite{GR86,J16}.  As $Re$ increases (i.e. $\kappa$ increases), the convergence of Picard slows, and for $Re$ large enough the method fails to converge \cite{PRTX25,PRX19}.  Unfortunately, this failure occurs for $Re$ well within the range of physically relevant problems \cite{Laytonbook}. 
To improve the speed and robustness of Picard for the NSE, incorporating Anderson acceleration (AA) was proposed in \cite{PRX19} and was found in \cite{PRX19,PR21,PR23} to significantly accelerate a converging Picard iteration as well as dramatically increase the range of $Re$ for which convergence could be achieved.  Moreover, it was rigorously proven in \cite{PRX19} that AA improves the linear convergence rate of Picard by the gain of the AA optimization problem (details in section 2), giving theory to the improvements seen in computational tests.    Of course, other techniques have been proposed over the years to improve the Picard iteration for NSE and related systems such as \cite{ES96,FMW19} and many others, but what makes AA especially interesting and useful is its simple implementation, its typically negligible computational cost, and its ability to be effective on a wide range of problems; more AA background on AA is given in Section 2.  We note also that AA-Picard was also shown to be an excellent nonlinear preconditioner for Newton solvers for NSE \cite{MXZ25,PRTX25}, and our work herein applies to AA-Picard as both a standalone solver and as a nonlinear preconditioner.

Despite AA being developed by D.G. Anderson in 1965 \cite{Anderson65}, the 2019 result of \cite{PRX19} that proved AA accelerates the Picard iteration for the NSE was the first time it was proven AA actually accelerated convergence of a fixed point iteration for a nonlinear problem.  The results in \cite{PRX19} were then generalized to contractive fixed point iterations in \cite{EPRX20} and extended further with sharper estimates and even holding for noncontractive fixed point iterations in \cite{PR21}.  A key requirement in these results is that the AA optimization norm must match the Hilbert space norm the fixed point function is defined on (e.g. for Picard for NSE, the norm is $H^1_0$).  However, this presents a discrepancy between theory and practice because typically the $\ell^2$ norm is used in AA implementations.
There are several reasons practitioners choose $\ell^2$, including simplicity in implementation and in using major codes that have AA built in \cite{petsc-web-page,gardner2022sundials}, using legacy codes since these are older than the theory, but perhaps mostly because AA using $\ell^2$ has a long history of often working well.  
There is little research done regarding AA optimization norm choice, although one important contribution is the 2022 paper of Yang, Townsend and Apell\"o \cite{YTA22} that showed for certain classes of PDEs (e.g. nonlinear Helmholtz equation), using $H^{-s}$ for the optimization norm can provide significantly better convergence results than using $\ell^2$.


The purpose of this paper is to analytically and numerically investigate the convergence of AA-Picard for the NSE with different choices of AA optimization norm.
 This paper has three main results.  First, we employ a problem-specific analysis using the $H^1_0$ norm for the AA optimization in which we improve and generalize known convergence results.  This approach allows us to take advantage of sharper nonlinear term estimates.  Additionally, our results improve on those from \cite{PR25,PRX19} by extending to general AA depth $m$, removing the small data restriction, and developing new AA identities that allow for a smooth analysis of the higher order terms.  Second, we extend these new results to the case of using $L^2$ as the AA optimization norm.  Here, we prove a convergence estimate that is observed (through a new representation of the AA gain in the convergence result) to be very similar to that of the $H^1_0$ case.  Third, we give results for three numerical tests using varying AA optimization norms that first illustrate our theory that AA-Picard performs about equally well using $H^1_0$ and $L^2$ optimization norms, but also show that on a common benchmark problem where AA-Picard 
 using $L^2$ and $H^1_0$ optimization norms perform well, when using $\ell^2$ the performance is significantly worse and even fails in one mesh.
 While our results are for a specific application problem, they have general implications since they demonstrate that even when the theory and tests show that using AA with the $L^2$ optimization norm is a good choice, one may get significantly worse results or even failure when using $\ell^2$ instead.  
 


This paper is arranged as follows.  Section 2 defines notation, gives mathematical preliminaries, and provides background on Navier-Stokes equations, the Picard iteration, Anderson acceleration, and AA-Picard.  Our new analysis is given in section 3, and numerical tests are found in section 4.  Finally, conclusions and future directions are discussed in section 5.

\section{Notation and Preliminaries}

We consider an open connected set $\Omega\subset\mathbb{R}^d$ (d=2 or 3) as the domain, and the $L^2(\Omega)$ inner product and norm are denoted as $(\cdot,\cdot)$ and $\|\cdot\|$, respectively.  Other norms will be clearly labeled with subscripts.  The notation  $\langle \cdot,\cdot\rangle$
is used to represent the duality pairing between $H^{-1}(\Omega)$ and $H^1_0(\Omega)$, and we use $\|\cdot\|_{-1}$ to denote the norm on $H^{-1}(\Omega)$.

Define the natural pressure and velocity function spaces for the NSE by 
	\begin{align*}
		&Q:=\{q\in {L}^2(\Omega): \int_{\Omega}q\ dx=0\},\\
		& X:=
		\{v\in H^1\left(\Omega\right): v=0~~\text{on}~ \partial\Omega\},
		\end{align*}
		along with the divergence-free velocity space
		\[
		 V:=
		\{v\in X:  (\nabla \cdot v,q)=0\ \forall q\in Q\}.
		\]

\color{black}
	
The Poincar\'e inequality holds on $X$: there exists a constant $C_P(\Omega)$ depending only on the domain size which satisfies 
\begin{align}\label{eqn:poincare}
\| \phi \| \le C_P \|\nabla \phi \| \mbox{ for all } \phi\in X.
\end{align}
		
We use the skew-symmetric form of the nonlinear term: for all $v,w,z\in X,$
\[
b^*(v,w,z) = (v\cdot\nabla w,z) + \frac12 ((\nabla \cdot v)w,z).
\]
One could also use other energy preserving formulations such as rotational and EMAC forms \cite{CHOR17} and get similar results as we find herein.

A key property of $b^*$ is that 
\[
b^*(v,w,w)=0 \ \forall w,v\in X.
\]
If the first argument of $b^*$ is divergence-free, i.e. it satisfies $\| \nabla \cdot v\|=0$, then skew-symmetry has no effect.   However, when using certain finite element subspaces of $X$ and $Q$ (such as Taylor-Hood elements), the discretely divergence-free space does not contain only divergence-free functions \cite{JLMNR17}.  Hence, to be general, we utilize skew-symmetry in our analysis.

Bounding the $b^*$ functional is critical to our residual convergence analysis.  The next lemma is known from e.g. \cite{temam,Laytonbook}, but we give a brief proof so that constants are fixed for later reference.

\begin{lemma} \label{bbounds}
Suppose $v,w,z\in X$.  Then there exists $M=M(\Omega)$ such that
\begin{align}
| b^*(v,w,z) | & \le M C_P^{-1/2}    \| v \|^{1/2} \| \nabla v \|^{1/2}  \| \nabla w \| \| \nabla z \|, \label{H1L2bound} \\
| b^*(v,w,z) | & \le M    \| \nabla v \|  \| \nabla w \| \| \nabla z \|, \label{bstarbound} 
\end{align}

\end{lemma}
\begin{proof}
We start with Green's theorem following \cite{J16} to obtain
\[
b^*(v,w,z) = \frac12 (v\cdot\nabla w,z) - \frac12 (v\cdot\nabla z,w).
\]
Next, using H\"older's inequality, we get that
\begin{align*}
b^*(v,w,z) &= \frac12 (v\cdot\nabla w,z) - \frac12 (v\cdot\nabla z,w) \\
& \le \frac12 \| v \|_{L^3} \| \nabla w \|_{L^2} \| z \|_{L^6} + \frac12 \| v \|_{L^{3}} \| \nabla z \|_{L^{2}} \| w \|_{L^{6}},
\end{align*}
and then applying a Sobolev inequality from the embedding of $H^1(\Omega)$ into $L^6(\Omega)$ produces
\begin{equation}
b^*(v,w,z) \le C_6 \| v \|_{L^3} \| \nabla w \| \| \nabla z \|. \label{L3bound}
\end{equation}
Next, using the Sobolev bound on $L^3(\Omega)$ \cite{BGHRR25}: $\| \phi \|_{L^3} \le C_3 \| \phi \|^{1/2} \| \nabla \phi \|^{1/2}$ and obtain
\begin{equation}
b^*(v,w,z) \le C_6 C_3 \| v \|^{1/2} \| \nabla v \|^{1/2}  \| \nabla w \| \| \nabla z \|. \label{H1L2boundb}
\end{equation}
An application of Poincar\`e to $\| v \|$ will produce the bound \eqref{bstarbound}, with $M=C_P^{1/2} C_6 C_3$.  Using this definition of $M$ in \eqref{H1L2boundb}, we establish \eqref{H1L2bound} as well.

\end{proof}

\subsection{Finite element preliminaries}

Let $\tau_h(\Omega)$ be a conforming mesh and $X_h(\tau_h) \times Q_h(\tau_h) \subset X \times Q$
be conforming finite element spaces for the velocity and pressure, respectively.  We require the pair $(X_h,Q_h)$ to satisfy  the inf-sup condition
\[
\sup_{0\neq  v \in X_h} \frac{(\nab\cdot  v ,q)}{\|\nabla  v \|}\ge \beta \|q\|\qquad \forall q\in Q_h,
\]
for some $\beta>0$ that is independent of $h$.  Commonly used examples are the Taylor-Hood and Scott-Vogelius elements, with the latter possibly requiring meshes to have a particular macro-element structure depending on the polynomial degree \cite{JLMNR17,GS19,arnold:qin:scott:vogelius:2D,Z05}.

Define the discretely divergence-free space by 
\[
	 V_h:=
		\{v \in X_h:  (\nabla \cdot v,q)=0\ \forall q\in Q_h\}.
\]

\subsection{NSE preliminaries}

The weak form of the NSE \eqref{NS} is given by \cite{Laytonbook}: Find $u\in V$ satisfying 
\begin{equation}
\nu(\nabla u,\nabla v) + b^*(u,u,v) = \langle f,v \rangle\ \forall v\in V. \label{weakNS}
\end{equation}

Solutions exist for the weak steady NSE system \eqref{weakNS}  for any $f\in H^{-1}(\Omega)$ and $\nu>0$ \cite{GR86,Laytonbook}, and any such solution satisfies
\begin{equation}
\| \nabla u \| \le \nu^{-1} \| f \|_{-1}. \label{nsstab}
\end{equation}
Hence we assume throughout that $\nu>0$ and $f\in H^{-1}(\Omega)$.

A sufficient condition for uniqueness of solutions to \eqref{weakNS} is that the data satisfy the smallness condition $\kappa:=M\nu^{-2} \|f \|_{-1}<1$ \cite{GR86,Laytonbook}.
While $\kappa<1$ is not necessary for uniqueness, it is known that for sufficiently large data (i.e. $\kappa$ large enough), uniqueness breaks down and the NSE will admit multiple solutions \cite{Laytonbook}.  Note that we do not assume $\kappa<1$ in our analysis.

The finite element NSE formulation is given by: Find $u\in V_h$ satisfying 
\begin{equation}
\nu(\nabla u,\nabla v) + b^*(u,u,v) = \langle f,v \rangle\ \forall v\in V_h. \label{FEMNS}
\end{equation}
The same results as above for boundedness and well-posedness that hold in the $V$-formulation will also hold in the $V_h$ formulation, despite $V_h$ not necessarily being a subset of $V$ \cite{Laytonbook}.

\subsection{Picard for NSE preliminaries}

The finite element form of the Picard iteration for the NSE is: Given $u_k\in V_h$, find $u_{k+1}\in V_h$ satisfying 
\begin{equation}
\nu(\nabla u_{k+1},\nabla v) + b^*(u_k,u_{k+1},v) =\langle f,v \rangle\ \forall v\in V_h. \label{FEMPic}
\end{equation}
It is shown in \cite{PRX19} that the solution operator $g:V_h\rightarrow V_h$ associated with \eqref{FEMPic} is well-defined for any $\nu>0$ and $f\in H^{-1}(\Omega)$, and is uniformly bounded by
\begin{equation}
\| \nabla g(u_k) \|  \le \nu^{-1} \| f\|_{-1}. \label{Picbound}
\end{equation}  
Thus we can write \eqref{FEMPic} as the fixed point iteration $u_{k+1}=g(u_k)$.

It is also known that the error  in Picard satisfies \cite{GR86,PRX19}
\begin{equation}
\| \nabla (u-u_{k+1} ) \| \le \kappa \| \nabla (u-u_k) \|, \label{Picerr}
\end{equation}
and the fixed point residual satisfies
\begin{equation}
\| \nabla (g (u_{k})-u_{k} ) \| \le \kappa \| \nabla ( g(u_{k-1})-u_{k-1}) \|. \label{Picresid}
\end{equation}
If $\kappa<1$, \eqref{Picerr} and \eqref{Picresid} show that the Picard iteration will convergence linearly with rate (at least)  $\kappa$ to the unique weak NSE solution.  For $\kappa$ sufficiently large, the Picard iteration for the NSE will (generally) not converge \cite{PR25}.

We recall the following result from \cite{PR25} about smoothness properties of $g$.
\begin{lemma}[Lemma 6.2 from \cite{PR25}]\label{gproperties}
The operator $g:V_h\rightarrow V_h$ defined as the solution operator to \eqref{FEMPic} is Lipschitz continuously Frechet differentiable and satisfies
\[
\| \nabla g \| \le \nu^{-1} \| f\|_{-1},\ \ \| \nabla g' \| \le \kappa,
\]
where $g'$ is the Frechet derivative of $g$.

\end{lemma}

\subsection{Anderson acceleration preliminaries}

AA is defined as follows, for a given fixed point function $G:H\rightarrow H$ for a Hilbert space $H$ and optimization in a norm denoted $\| \cdot \|_*$. 
\begin{alg}[Anderson acceleration \cite{Anderson65}] \label{alg:anderson}
Anderson acceleration with depth $m$ reads: \\ 
Step 0: Choose $u_0\in H.$\\
Step 1: Find $\tilde u_1\in H $ such that $\tilde u_1 = G(u_0)$.  
Set $u_1 = \tilde u_1$. \\
Step $k$: For $k+1=1,2,3,\ldots$ Set $m_k = \min\{ k, m\}.$\\
\indent [a.] Find $\tilde u_{k+1} = G(u_k)$. \\
\indent [b.] Solve the minimization problem for $\{ \alpha_{j}^{k+1}\}_{k-m_k}^k$
\[
\min_{\sum\limits_{j=k-m_k}^{k} \alpha_j^{k+1}  = 1} 
\left\|   \sum\limits_{j=k-m_k}^{k} \alpha_j^{k+1}( \tilde u_{j+1} - u_j) \right\|_* .
\] 
\indent [c.] Set  $$u_{k+1} = \beta_{k+1} \sum\limits_{j= k-m_k}^k 
\alpha_j^{k+1} \tilde u_{j+1} + (1-\beta_{k+1}) 
\sum\limits_{j = k-m_k}^k \alpha_j^{k+1} u_j,$$  
for damping parameter $0 < \beta_{k} \le 1$.
\end{alg}

While the general AA algorithm allows for adaptive depth $m_k$, for simplicity of our single-step analysis we will assume a fixed depth $m$. For simplicity, we also assume no relaxation in this paper ($\beta_k=1$).  We note that AA is often equivalently reformulated to use an unconstrained optimization problem \cite{PR25}.


An implementation of AA is given as follows \cite{PR25}.  Define $F$ to be the rectangular matrix with columns  $\{ \tilde u_{j+1} - u_j \}_{j=k-m}^k$, and the vector $\alpha=\{ \alpha_j \}_{j=k-m}^k$.  Then the optimization problem can be written equivalently as
\[
\min_{\sum\limits_{j=k-m}^{k} \alpha_j^{k+1}  = 1} 
\left\|   \sum\limits_{j=k-m}^{k} \alpha_j^{k+1}( \tilde u_{j+1} - u_j) \right\|_* \iff \min_{\sum\limits_{j=k-m}^{k} \alpha_j^{k+1}  = 1}  \left\| F\alpha    \right\|_*^2
\] 
In the finite element setting, $ \| F \alpha \|_X^2 = \alpha^T F^T M F \alpha$, where for example $M$ is the mass matrix if $*=L^2(\Omega)$, the stiffness matrix if $*=H^1_0(\Omega)$, or the identity matrix if $*=\ell^2$.

Next, define the reduced $QR=M^{1/2} F$, so that
\[
F^T  M F = R^T Q^T Q R = R^T R,
\]
and note that $M^{1/2}$ need not be computed, it only needs to exist which it will if $M$ is symmetric positive definite.  From here, a Cholesky factorization of $F^T  M F$ produces $R$, so that $\alpha^T F^T  M F \alpha = \| R \alpha \|_{2}^2$.  Now setting $R\alpha=0$ and directly inserting the constraint $\sum\limits_{j=k-m}^{k} \alpha_j^{k+1}  = 1$ gives
a $(m +1) \times m$ linear least squares problem that can be easily solved for any reasonable $m$.

AA was originally introduced in 1965 by D.G. Anderson \cite{Anderson65}, and its use exploded in the scientific community after the Walker and Ni 2011 paper \cite{WaNi11} that showed AA is a simple way to improve many types of nonlinear solvers.  AA has recently been applied to thousands of different problems, and
if one peruses the over 830 citations of \cite{WaNi11}, it is hard to find a nonlinear system that AA is not being used on.  The first convergence theory for AA was established by Toth and Kelley in 2015 in the paper \cite{ToKe15} and then sharpened in \cite{K18}, and essentially proved that AA would do no harm to a converging fixed point iteration.

The first convergence theory for AA that showed it {\it accelerated} a nonlinear solver was the paper on AA-Picard for the NSE \cite{PRX19} which we extend results for herein.  Convergence results for general contractive \cite{EPRX20} and noncontractive \cite{PR21} cases came as extensions of \cite{PR21}, sharpening estimates along the way.  AA has also been built into widely used software such at PETSc and SUNDIALS \cite{petsc-web-page,gardner2022sundials} and through examples in deal.ii \cite{D2}.
For more on AA and its history, see the book \cite{PR25} and the excellent review papers \cite{K18,S24}.

\subsection{AA-Picard for NSE preliminaries}

The AA-Picard iteration for the discrete incompressible steady NSE  can now be written as Algorithm \ref{alg:anderson} with $G=g$, where $g$ is the solution operator of \eqref{FEMPic}.  For simplicity, we consider the case of no relaxation, $\beta_k =1$ and constant depth $m$.  Adding relaxation to our results can be performed following analysis in \cite{PR21}, and modifies the results in the expected way.

The following theorem was proven in \cite{PRX19} for AA-Picard for NSE with $m=1$ and $\kappa<1$.  We denote the nonlinear residual by $w_{k} := g(u_{k-1}) - u_k = \tilde u_{k} - u_{k-1}$.  The optimization problem in \cite{PRX19} used the natural Hilbert space norm of the problem, $H^1_0$, and reads: find $\{ \alpha_{j}^{k+1}\}_{k-1}^k$ satisfying
\[
\min_{\sum\limits_{j=k-1}^{k} \alpha_j^{k+1}  = 1} 
\left\|   \nabla \left( \sum\limits_{j=k-1}^{k} \alpha_j^{k+1}w_j \right) \right\| ,
\] 
and we note we have shifted the indexing by one from \cite{PRX19}, to match that of the more recent AA theory paper \cite{PR21}.  The `gain of the optimization problem,' a term coined in \cite{PRX19}, is then given by
\[
\theta_{k,1}^{H^1_0} = \frac{ \| \nabla (\alpha_{k}^k w_{k} + \alpha_{k-1}^k w_{k-1}) \| } {\| \nabla w_{k} \|}.
\]
Note that $0 \le \theta_{k,1}^{H^1_0} \le  1$, and the only way $\theta_{k,1}^{H^1_0}=1$ is in the (unlikely) case that the optimization cannot do better than the standard Picard step (i.e. when $\alpha^k_{k}=1$ and $\alpha^k_{k-1}=0$).  Hence $\theta_{k,1}^{H^1_0}$ can be considered the mechanism through which AA improves convergence, as is demonstrated in the following result from \cite{PRX19}.

\begin{theorem}[Convergence of the AA-Picard for NSE residual with $m=1$ under a small data assumption \cite{PRX19}] 
\label{thm:m1}
Suppose $\kappa<1$, the $H^1_0$ norm is used for the AA optimization, and $0 < |\alpha_{k-1}^k| < \bar \alpha$ for some fixed $\bar \alpha$.
Then on any step where $\alpha_{k-1}^k \ne 0$, 
the $m=1$ AA-Picard iterates satisfy
\begin{equation}\label{eqn:ck-res}
\| \nabla w_{k+1} \| \le \theta_{k,1}^{H^1_0} \kappa \|\nabla w_{k} \|
+\frac{M \bar\alpha}{\nu (1-\kappa)^2}  \|\nabla w_{k} \| \nr{\grad w_{k-1}}.
\end{equation}
\end{theorem}

The paper \cite{PRX19} extended Theorem \ref{thm:m1} to $m=2$, but did not give results for larger $m$, and also does not give results without the assumption of small problem data.  However, utilizing AA theory from \cite{PR21}, a more general result can be easily constructed.  For this, we first define a depth $m$ gain of the optimization problem
\[
\theta_{k,m}^{H^1_0} = \frac{ \| \nabla \sum_{j=k-m}^{k} \alpha_{j}^k w_{j}) \| } {\| \nabla w_{k} \|}.
\]
An important assumption for extending AA results outside of the contractive case involves a relationship between $e_j = u_j - u_{j-1}$ and $w_j$'s and comes from \cite{PR21}.
\begin{assumption}\label{assume}
The stage $j$ iterates and residuals of depth $m$ AA-Picard are assumed to satisfy 
\[
\sigma \| \nabla e_j \| \le \| \nabla (w_{j+1} - w_{j} ) \|.
\]
\end{assumption}
This assumption is automatically satisfied if $g$ is contractive, and it also holds for a large class of noncontractive $g$ (see Remark 2.1 of \cite{PR21}). It is shown in \cite{PR21} that this assumption implies that
\begin{equation}
 \| \nabla e_j \| \le K \| \nabla w_{j} \|, \label{ewbound}
\end{equation}
where $K$ depends on $\sigma$, $\theta_k^{H^1_0}$, and the degree to which the elements of $\{ (w_{k} - w_{k-1}), ... , (w_{k-m}-w_{k-m-1}) \}$ are linearly independent with respect to the $H^1_0$ inner product (i.e. $K$ blows up as linear independence is lost, or as $\sigma$ goes to zero).  Herein, the $K$ arises in the convergence analysis only in the higher order terms.

With Assumption \ref{assume} and the smoothness properties of $g$ from Lemma \ref{gproperties}, convergence of AA-Picard was proven for general $m$ in \cite{PR25} and we state the result now.  As our focus herein is on local convergence, we do not get into the technical details of the higher order terms and being more precise with the constant $C$ in the following theorem; the interested reader is referred to \cite{PR21} for more on this topic.

\begin{theorem}[\cite{PR25}] \label{thmm}
Under Assumption \ref{assume} and the smoothness properties of $g$ given in Lemma \ref{gproperties}, the $k^{th}$ step residual of depth $m$ AA-Picard using $H^1_0$ as the optimization norm satisfies
\[
\| \nabla w_{k+1} \| \le \theta_{k,m}^{H^1_0}  \kappa \| \nabla w_k \| + CK \kappa \| \nabla w_k\| \left( \sum_{j=k-m}^{k-1} \| \nabla w_j \| \right).
\]
\end{theorem}

\section{Improved convergence analysis and alternative optimization norms}

We now consider improving the convergence results for AA-Picard for NSE.  To do this, we use a problem-specific approach that allows us to take advantage of a sharper treatment of the nonlinear terms in order to improve estimates of the first order terms in the convergence analysis, but still exploit 
 ideas from \cite{PR21} for the higher order terms.  
%
We begin this section by first improving on Theorems \ref{thm:m1} and \ref{thmm} for $m=1$ AA-Picard. With our new approach, we gain an improvement in the first order terms of the residual convergence estimate and also remove the small data assumption used in \cite{PRX19}.  We then extend these results to the case of general $m$.  The proof for general $m$ is significantly more technical and requires some new AA identities in the  nonlinear terms, but follows the same general structure as the $m=1$ case.  Lastly, we extend our results to general $m$ AA-Picard that uses the $L^2$ optimization norm.  This last result has first order terms in the convergence estimate that turn out to be very similar to what is found using the $H^1_0$ norm, which is perhaps not expected since $H^1_0$ is the associated Hilbert space norm.
Hence our new theory shows that both $H^1_0$ and $L^2$ are good (if not correct) norm choices.

Due to the use of the $L^2$ norm in the optimization, it is necessary to define the gain of the $L^2$ optimization problem:
\begin{equation}
\theta_{k,m}^{L^2} = \frac{ \| \sum_{j=k-m}^k \hat\alpha_{j}^k w_{j} \| } {\| w_{k} \|},  \label{thetaL2}
\end{equation}
where the $\hat \alpha_j$'s are the optimal $L^2$ coefficients.  Note the notational difference between the optimal $H^1_0$ coefficients (no hats) and the optimal $L^2$ coefficients (yes hats).

%

The following identity will be used in the analysis below.  Note that the same result holds with $\hat\alpha$'s replacing all the $\alpha$'s.
\begin{lemma}\label{weidentities}
The following identity holds for general $m$ relating $e_j$'s and $w_i$'s:
\[
\sum_{j=k-m}^k  \alpha^k_j w_j
=
e_k +  (1-\alpha^k_k)  e_{k-1}  +   (1-\alpha^k_k - \alpha^k_{k-1}) e_{k-2} + ... +  \left(1- \sum_{j=k-m+1}^k \alpha^k_j \right) e_{k-m}.
\]
\end{lemma}
\begin{proof}[Proof of Lemma \ref{weidentities}]

We begin by expanding $\sum_{j=k-m}^k \alpha^k_j w_j$ by using the definition of the $w_j$'s:
\begin{align*}
\sum_{j=k-m}^k & \alpha^k_j w_j
 = \alpha^k_k w_k + \alpha^k_{k-1} w_{k-1} + ... + \alpha^k_{k-m} w_{k-m}\\
& = \left( \alpha^k_k \tilde u_k + \alpha^k_{k-1} \tilde u_{k-1} + ... + \alpha^k_{k-m} \tilde u_{k-m} \right) - \left( \alpha^k_k  u_{k-1} + \alpha^k_{k-1} u_{k-2} + ... + \alpha^k_{k-m} u_{k-m-1} \right)  \\
& =u_k - \left( \alpha^k_k  u_{k-1} + \alpha^k_{k-1} u_{k-2} + ... + \alpha^k_{k-m} u_{k-m-1} \right),
\end{align*}
with the last step using the definition of $u_{k}$.  Next, add and subtract $u_{k-1}$ to the right hand side and use the definition of $e_k$ to get
\[
\sum_{j=k-m}^k  \alpha^k_j w_j
=
e_k + u_{k-1} - \alpha^k_k  u_{k-1} - \alpha^k_{k-1} u_{k-2} - ... -  \alpha^k_{k-m} u_{k-m-1}.
\]
Since $1-\alpha^k_k = \sum_{j=1}^{k-1} \alpha^k_j$, we write the equation as
\[
\sum_{j=k-m}^k  \alpha^k_j w_j
=
e_k +  \sum_{j=k-m}^{k-1} \alpha^k_j u_{k-1}  - \alpha^k_{k-1} u_{k-2} - ... -  \alpha^k_{k-m} u_{k-m-1}.
\]
Next, add and subtract $ \sum_{j=k-m}^{k-2} \alpha^k_j u_{k-2}$ to get
\[
\sum_{j=k-m}^k  \alpha^k_j w_j
=
e_k +  \sum_{j=k-m}^{k-1} \alpha^k_j e_{k-1}  +   \sum_{j=k-m}^{k-2} u_{k-2} - ... -  \alpha^k_{k-m} u_{k-m-1}.
\]
Repeating this process by adding and subtracting each of $ \sum_{j=1}^{k-2}\alpha^k_j u_{k-3}$, $ \sum_{j=k-m}^{k-3}\alpha^k_j u_{k-4}$, ...,  $ \sum_{j=k-m}^{k-m+1}\alpha^k_j u_{k-m}$ to the right hand side gives us
\[
\sum_{j=k-m}^k  \alpha^k_j w_j
=
e_k +  \sum_{j=k-m}^{k-1} \alpha^k_j e_{k-1}  +   \sum_{j=k-m}^{k-2} \alpha^k_j e_{k-2} + ... +  \alpha^k_{k-m} e_{k-m}.
\]
Since  $1 = \sum_{j=1}^{k} \alpha^k_j$, we can equivalently write this equality as
\[
\sum_{j=k-m}^k  \alpha^k_j w_j
=
e_k +  (1-\alpha^k_k)  e_{k-1}  +   (1-\alpha^k_k - \alpha^k_{k-1}) e_{k-2} + ... +  \left(1- \sum_{j=k-m+1}^k \alpha^k_j \right) e_{k-m}.
\]
This completes the proof.

\end{proof}

\subsection{Convergence analysis for AA-Picard with $H^1_0$ optimization norm}

We now prove convergence results for AA-Picard using the $H^1_0$ optimization norm.  We first give the $m=1$ result, followed by the result for general $m$.

\begin{theorem}[Improved convergence result for $H^1_0$ norm in AA optimization] 
\label{thm:m1b}
Under Assumption \ref{assume} and assuming $0<|\alpha_j^k|<\bar{\alpha} \ \forall k,j$ for some fixed $\bar{\alpha}$ and that the residual differences $\{ w_j - w_{j-1} \}_{j=1,2,...,k}$ are not linearly dependent in the $H^1_0$ inner product, the depth $m=1$ AA-Picard iterates found using $H^1_0$ as the optimization norm satisfy
\begin{multline*}
\| \nabla w_{k+1}\| \le 
  \kappa {\theta_{k,1}^{L^2}}^{1/2} {\theta_{k,1}^{H^1_0}}^{1/2}
 \frac{\| \alpha^k_{k} w_{k} + \alpha^k_{k-1} w_{k-1} \|}{\| \hat \alpha^k_{k} w_{k} + \hat \alpha^k_{k} w_{k-1} \|} ^{1/2}  \left( \frac{\| w_{k} \|}{C_P \| \nabla w_{k} \| } \right)^{1/2}  \|\nabla  w_{k} \|
\\ +  K^2 M\nu^{-1}  \kappa  \bar{\alpha} \|\nabla w_{k}\| \| \nabla w_{k-1} \|  . 
\end{multline*}
\end{theorem}

\begin{remark}
In Theorem \ref{thm:m1b}, there are several equivalent ways to write the first order term coefficient of $\| \nabla w_k\|$.  For example, we can eliminate $\theta_{k,1}^{L^2}$ using its definition to get
\begin{equation}
 \kappa {\theta_{k,1}^{L^2}}^{1/2} {\theta_{k,1}^{H^1_0}}^{1/2}
 \frac{\| \alpha^k_{k} w_{k} + \alpha^k_{k-1} w_{k-1} \|}{\| \hat \alpha^k_{k} w_{k}+ \hat \alpha^k_{k-1} w_{k-1} \|} ^{1/2}  \left( \frac{\| w_{k} \|}{C_P \| \nabla w_{k} \| } \right)^{1/2} 
 =
 \kappa {\theta_{k,1}^{H^1_0}}^{1/2}
 \left( \frac{\| \alpha^k_{k} w_{k} + \alpha^k_{k-1} w_{k-1} \|}{C_P \| \nabla w_{k} \| }\right)   ^{1/2}. \label{cvbound1}
\end{equation}
We could also upper bound the numerator on the right hand side of \eqref{cvbound1} using the Poincar\'e inequality to recover the (less accurate) first order estimate of Theorems \ref{thm:m1} and \ref{thmm}:
\[
 \kappa {\theta_{k,1}^{H^1_0}}^{1/2}
 \left( \frac{\| \alpha^k_{k} w_{k} + \alpha^k_{k-1} w_{k-1} \|}{C_P \| \nabla w_{k} \| }\right)   ^{1/2}
\le \kappa \theta_{k,1}^{H^1_0}.
\]

We choose to write the first order term as in Theorem \ref{thm:m1b} because it provides for an easy comparison to the estimate to follow for the case of the $L^2$ norm used in the AA optimization.
\end{remark}

\begin{remark}
Comparing Theorem \ref{thm:m1b} to Theorems \ref{thm:m1} and \ref{thmm} (with depth $m=1$), the key difference for local convergence is in the first order terms.  While Theorems \ref{thm:m1} and \ref{thmm} scale $\| \nabla w_k \|$ by $\kappa \theta_{k,1}^{H^1_0}$, this result instead scales it by
\[
\kappa {\theta_{k,1}^{L^2}}^{1/2} {\theta_{k,1}^{H^1_0}}^{1/2}
 \frac{\| \alpha^k_{k} w_{k} + \alpha^k_{k-1} w_{k-1} \|}{\| \hat \alpha^k_{k} w_{k} + \hat \alpha^k_{k} w_{k-1} \|} ^{1/2}  \left( \frac{\| w_{k} \|}{C_P \| \nabla w_{k} \| } \right)^{1/2}.
\]
From our numerical testing, the change from $\theta_{k,1}^{H^1_0}$ to $ {\theta_{k,1}^{L^2}}^{1/2} {\theta_{k,1}^{H^1_0}}^{1/2}$ is not consequential (we observed only minor differences in the calculations of these terms).  We also found the term $ \frac{\| \alpha^k_{k} w_{k-1} + \alpha^k_{k-1} w_{k-2} \|}{\| \hat \alpha^k_{k} w_{k-1} + \hat \alpha^k_{k-1} w_{k-2} \|} ^{1/2}$ to not make significant differences numerically; while this ratio is at least 1 (since $\hat\alpha$'s are the optimal $L^2$ coefficients), we find in numerical tests it is generally close to 1.  The term $\left( \frac{\| w_{k-1} \|}{C_P \| \nabla w_{k-1} \| } \right)^{1/2}$, on the other hand, does make a significant difference in our tests: $\| w_{k-1}\|$ is typically an order of magnitude smaller than $\| \nabla w_{k-1}\|$.  Hence  it is this term that makes the first order terms in the bound  of Theorem \ref{thm:m1b} (generally) a significant improvement over those in Theorems \ref{thm:m1} and \ref{thmm}.
\end{remark}

\begin{proof}[Proof of Theorem \ref{thm:m1b}]
The proof begins by defining $e_k, \te_k$ and $w_k$ by 
\begin{align*}
e_k & := u_k - u_{k-1}, \\
\te_k & := \tu_k - \tu_{k-1}, \\
w_k &= \tu_{k} - u_{k-1}.
\end{align*}
The Picard scheme at $j=k-1$ and $j=k$ gives us two equations for finding $\tilde u_k,\ \tilde u_{k-1} \in V_h$, respectively: for all $v\in V_h$,
\begin{align}
\nu (\nabla \tilde u_k,\nabla v) + b^*(u_{k-1},\tilde u_k,v) & = \langle f,v \rangle, \label{Pk} \\
\nu (\nabla \tilde u_{k-1},\nabla v) + b^*(u_{k-2},\tilde u_{k-1},v) & = \langle f,v \rangle. \label{Pkm1}
\end{align}
Multiplying \eqref{Pk} by $\alpha^k_{k}$ and \eqref{Pkm1} by $\alpha^k_{k-1}$ and then adding the equations gives for all $v\in V_h$,
\[
\nu (\nabla u_k,\nabla v) + b^*(u_{k-1},\alpha^k_{k} \tilde u_k,v) + b^*(u_{k-2},\alpha^k_{k-1} \tilde u_{k-1},v)   = \langle f,v \rangle,
\]
by utilizing that $u_k = \alpha^k_{k} \tilde u_k + \alpha^k_{k-1}\tilde u_{k-1}$ and $ \alpha^k_{k} +  \alpha^k_{k-1}=1$.  For the second and third left hand side terms, we use the identity $ab+cd = -(a-c)d + a(b+d)$ to get for all $v\in V_h$,
\begin{equation}
\nu (\nabla u_k,\nabla v) - b^*(e_{k-1},\alpha^k_{k-1} \tilde u_{k-1},v) + b^*(u_{k-1},u_k,v) = \langle f,v \rangle. \label{ZZZ1}
\end{equation}
Subtracting this from the Picard scheme at $j=k+1$ (i.e. \eqref{Pk} but with each subscript increased by 1), we get for all $v\in V_h$,
\begin{equation}
\nu(\nabla (\tilde u_{k+1} - u_k),\nabla v) + b^*(u_{k},\tilde u_{k+1},v) - b^*(u_{k-1},u_k,v) + b^*(\alpha^k_{k-1} e_{k-1},\tilde u_{k-1},v)=0,
\end{equation}
noting we moved the scalar $\alpha^k_{k-1} $ to the first argument of the last $b^*$ term.  Using the identity $ab-cd = a(b-d) + (a-c)d $ on the second and third left hand side terms, we obtain the equation, for all $v\in V_h$,
\begin{equation}
\nu(\nabla w_{k+1},\nabla v) + b^*(u_{k},w_{k+1},v) + b^*(e_{k},u_k,v) + b^*(\alpha^k_{k-1} e_{k-1},\tilde u_{k-1},v)=0.
\end{equation}
Choosing $v = w_{k+1}$  vanishes the second term and after using the identities $ab+cd= a(b-d) + (a+c)d$ and $u_k - \tilde u_{k-1}=\alpha^k_k \tilde e_k$ (note the latter holds since we are in the $m=1$ case) provides us with
\begin{equation}
\nu  \| \nabla w_{k+1} \|^2  
= - b^*(e_k + \alpha^k_{k-1} e_{k-1}, \tilde u_{k-1},w_{k+1}) - b^*(e_k, \alpha^k_{k} \tilde e_k,w_{k+1}). \label{Z2}
\end{equation}
Now using the identity $e_k + \alpha^k_{k-1} e_{k-1}=\alpha^k_{k} w_{k} + \alpha^k_{k-1} w_{k-1}$ (which reveals itself immediately by expanding the $e_j$ and $w_j$ terms using their definitions) and bounding the first right hand side term using \eqref{H1L2bound} and the second with \eqref{bstarbound}, we get the estimate
\begin{multline*}
\| \nabla w_{k+1}\| \le 
M C_P^{-1/2} \nu^{-1} 
\| \alpha^k_{k} w_{k} + \alpha^k_{k-1} w_{k} \|^{1/2} \|\nabla ( \alpha^k_{k} w_{k} + \alpha^k_{k-1} w_{k}) \|^{1/2}  \| \nabla \tilde u_{k-1}\| 
\\ + M\nu^{-1} \kappa | \alpha^k_{k}|  \| \nabla e_{k} \|  \| \nabla e_{k-1} \|, 
\end{multline*}
thanks to $\|\nabla \tilde e_k \| \le \kappa \| \nabla e_{k-1} \|$ which holds due to the Lipschitz continuity of $g'$ by Lemma \ref{gproperties}.    Using that $\| \nabla \tilde u_{k-1}\| \le \nu^{-1}\|f\|_{-1}$ from \eqref{Picbound}, the definition of $\theta_{k,1}^{H^1_0}$, and \eqref{ewbound}, we get 
\begin{multline*}
\| \nabla w_{k+1}\| \le 
 C_P^{-1/2} \kappa
\| \alpha^k_{k} w_{k} + \alpha^k_{k-1} w_{k-1} \|^{1/2} {\theta_{k,1}^{H^1_0}}^{1/2} \|\nabla  w_{k} \|^{1/2}  
 +  K^2 M \nu^{-1} \kappa  \bar{\alpha} \|\nabla w_{k}\| \| \nabla w_{k-1} \|. 
\end{multline*}
Next, multiply and divide the first right hand side term by $\| \hat \alpha^k_{k} w_{k} + \hat \alpha^k_{k-1} w_{k} \|^{1/2}$ and use the definition of $\theta_{k,1}^{L^2}$ to obtain
\begin{multline*}
\| \nabla w_{k+1}\| \le 
 C_P^{-1/2} \kappa
\left( \frac{\| \alpha^k_{k} w_{k} + \alpha^k_{k-1} w_{k-1} \|}{\| \hat \alpha^k_{k} w_{k} + \hat \alpha^k_{k-1} w_{k-1} \|}\right) ^{1/2} {\theta_{k,1}^{L^2}}^{1/2} \| w_{k} \|^{1/2} {\theta_{k,1}^{H^1_0}}^{1/2} \|\nabla  w_{k} \|^{1/2}  
\\ +  K^2 M \nu^{-1}  \kappa  \bar{\alpha} \|\nabla w_{k}\| \| \nabla w_{k-1} \|  . 
\end{multline*}
Finally, multiplying and dividing by  $\|\nabla w_{k} \|^{1/2}$ and rearranging terms produces the result of the theorem:
\begin{multline*}
\| \nabla w_{k+1}\| \le 
  \kappa {\theta_{k,1}^{L^2}}^{1/2} {\theta_{k,1}^{H^1_0}}^{1/2}
\left(  \frac{\| \alpha^k_{k} w_{k} + \alpha^k_{k-1} w_{k-1} \|}{\| \hat \alpha^k_{k} w_{k} + \hat \alpha^k_{k} w_{k-1} \|}\right) ^{1/2}  \left( \frac{\| w_{k} \|}{C_P \| \nabla w_{k} \| } \right)^{1/2}  \|\nabla  w_{k} \|
\\ +  K^2 M\nu^{-1}  \kappa  \bar{\alpha} \|\nabla w_{k}\| \| \nabla w_{k-1} \|  . 
\end{multline*}
\end{proof}

We now extend this improved result to AA-Picard with general $m$.  The proof structure is similar to that of the $m=1$ case, but there are significantly more terms and thus technicalities with the identities needed.  The identities used in the proof are lengthy, and so to keep the proof as clean as possible we move the identity derivations to the Appendix.

\begin{theorem}\label{thmH10m}
Under Assumption \ref{assume} and assuming $0<|\alpha_j^k|<\bar{\alpha} \ \forall k,j$ for some fixed $\bar{\alpha}$ and that the residual differences $\{ w_j - w_{j-1} \}_{j=1,2,...,k}$ are not linearly dependent in the $H^1_0$ inner product, the depth $m$ AA-Picard iterates found using $H^1_0$ as the optimization norm satisfy
\begin{multline}
\| \nabla w_{k+1} \| \le
\kappa {\theta_{k,m}^{H^1_0}}^{1/2} {\theta_{k,m}^{L^2}}^{1/2} \left(  \frac{ \| \sum_{j=k-m}^k \alpha^k_j w_j \| }{ \| \sum_{j=k-m}^k \hat \alpha^k_j w_j \|  }  \frac{ \| w_k \| }{ C_P \| \nabla w_k\|}\right)^{1/2} \| \nabla w_k \|\\
+
K^2 M \nu^{-1} \kappa \frac{m(m+3)\bar\alpha}{4} \sum_{j=k-m+1}^k \sum_{i=k-m}^{j-1} \| \nabla w_j \| \| \nabla w_i \|.
  \label{picm0}
\end{multline}
\end{theorem}
%
\begin{remark}
Just as in the $m=1$ case, the gain in the first order terms over Theorem \ref{thmm} comes from the term $ \frac{ \| w_k \| }{ C_P \| \nabla w_k\|}^{1/2}$, as we find  $\| w_k \|$ is typically an order of magnitude smaller than $\| \nabla w_k \|$ in our numerical tests.
\end{remark}
\begin{remark}
There are multiple ways to equivalently write the first order terms in \eqref{picm0}, and we choose this way because it allows for a straightforward comparison to the convergence estimate below when the $L^2$ norm is used in the AA optimization.
\end{remark}

\begin{proof}
We begin the proof analogous to the $m=1$ case, by multiplying the Picard scheme at $j=k,\ k-1, \ ..., k-m$ by $\alpha^k_k,\ \alpha^k_{k-1},\ ..., \alpha^k_{k-m}$, respectively, and summing all $m+1$ equations to get 
\begin{equation}
\nu(\nabla u_k,\nabla v) + \sum_{j=k-m}^k b^*(u_{j-1},\alpha^k_j \tilde u_j,v) = \langle f,v \rangle \ \forall v\in V_h. \label{group0}
\end{equation}

Before combining this equation with the $k+1$ level Picard equation to develop an equation for $w_{k+1}$, we rewrite the nonlinear terms in a particular way so that can be more easily analyzed.  Grouping the last two terms of the sum in \eqref{group0} using $ab+cd = -(a-c)d + a(b+d)$ gives
\begin{multline}
b^*(u_{k-m},\alpha^k_{k-m+1}\tilde u_{k-m+1},v)
+
b^*(u_{k-m-1},\alpha^k_{k-m}\tilde u_{k-m},v)
= \\
-b^*(e_{k-m},\alpha^k_{k-m}\tilde u_{k-m},v)
+
b^*(u_{k-m},\alpha^k_{k-m}\tilde u_{k-m} +\alpha^k_{k-m+1}\tilde u_{k-m+1} ,v). \label{group1}
\end{multline}
Adding in to \eqref{group1} the third to last term of the sum in \eqref{group0} gives, after using $ab+cd = -(a-c)d + a(b+d)$, that
\begin{multline}
 \sum_{j=k-m}^{k-m+2} b^*(u_{j-1},\alpha^k_j \tilde u_j,v)
=
-b^*(e_{k-m},\alpha^k_{k-m}\tilde u_{k-m},v) \\
+ b^*(u_{k-m+1},\alpha^k_{k-m+2}\tilde u_{k-m+2},v)
+b^*(u_{k-m},\alpha^k_{k-m}\tilde u_{k-m} +\alpha^k_{k-m+1}\tilde u_{k-m+1} ,v)
 \\
= -b^*(e_{k-m},\alpha^k_{k-m}\tilde u_{k-m},v) \\
-b^*(e_{k-m+1},\alpha^k_{k-m}\tilde u_{k-m} +\alpha^k_{k-m+1}\tilde u_{k-m+1} ,v)
+  b^*(u_{k-m+1},\sum_{j=k-m}^{k-m+2}  \alpha^k_j  \tilde u_j ,v).
\end{multline}
Repeating this process to eventually add in the rest of the terms from the sum in \eqref{group0} provides us with
\begin{multline}
\sum_{j=k-m}^k b^*(u_{j-1},\alpha^k_j \tilde u_j,v)
= b^*(u_{k-1},\sum_{j=k-m}^{k}  \alpha^k_j  \tilde u_j ,v) \\
-b^*(e_{k-m},\alpha^k_{k-m}\tilde u_{k-m},v)
-b^*(e_{k-m+1}, \sum_{j=k-m}^{k-m+1} \alpha^k_j \tilde u_j ,v) - ... - b^*(e_{k-1}, \sum_{j=k-m}^{k-1} \alpha^k_j \tilde u_j ,v).
\end{multline}
Now using the definition of $u_k$ and reducing gives
\begin{multline}
\sum_{j=k-m}^k b^*(u_{j-1},\alpha^k_j \tilde u_j,v)
= b^*(u_{k-1},u_k ,v) \\
-b^*(e_{k-m},\alpha^k_{k-m}\tilde u_{k-m},v)
-b^*(e_{k-m+1}, \sum_{j=k-m}^{k-m+1} \alpha^k_j \tilde u_j ,v) - ... - b^*(e_{k-1}, \sum_{j=k-m}^{k-1} \alpha^k_j \tilde u_j ,v). \label{group5}
\end{multline}
Combining \eqref{group5} with \eqref{group0} provides us with the equation
\begin{multline}
\nu(\nabla u_k,\nabla v) + b^*(u_{k-1},u_k ,v) -b^*(e_{k-m},\alpha^k_{k-m}\tilde u_{k-m},v) \\
-b^*(e_{k-m+1}, \sum_{j=k-m}^{k-m+1} \alpha^k_j \tilde u_j ,v) - ... - b^*(e_{k-1}, \sum_{j=k-m}^{k-1} \alpha^k_j \tilde u_j ,v) = \langle f,v \rangle \ \forall v\in V_h. \label{picm1}
\end{multline}

Next, proceeding as in the $m=1$ case by subtracting \eqref{picm1} from the step $k+1$ Picard scheme (i.e. \eqref{Pk} but with each subscript increased by 1), we get for all $v\in V_h$,
\begin{multline}
\nu(\nabla w_{k+1},\nabla v) + b^*(u_k,w_{k+1},v) + b^*(e_k,u_k,v) + b^*(e_{k-m},\alpha^k_{k-m}\tilde u_{k-m},v)  \\
+ b^*(e_{k-m+1}, \sum_{j=k-m}^{k-m+1} \alpha^k_j \tilde u_j ,v) + ... + b^*(e_{k-1}, \sum_{j=k-m}^{k-1} \alpha^k_j \tilde u_j ,v) = 0. \label{picm2}
\end{multline}
Before choosing $v=w_{k+1}$ and analyzing the residuals, it is helpful to apply some identities to the $b^*$ terms in \eqref{picm2} so they are in a form more acceptable for analysis.  By using that $ b^*(e_k,u_k,v) =  b^*(e_k,\sum_{j=k-m}^{k} \alpha^k_j \tilde u_j ,v)$ by expanding $u_k$, and then regrouping the third through the last terms in \eqref{picm2}, we get the equation
\begin{multline}
\nu(\nabla w_{k+1},\nabla v) + b^*(u_k,w_{k+1},v) + b^*(\alpha^k_k e_k,\tilde u_k,v) + b^*( \alpha^k_{k-1}(e_k + e_{k-1}),\tilde u_{k-1},v) \\
+ 
b^*( \alpha^k_{k-2}(e_k + e_{k-1} + e_{k-2}),\tilde u_{k-2},v)
+ ... +
b^*( \alpha^k_{k-m} \sum_{j=k-m}^k e_j ,\tilde u_{k-m},v)=0 \ \forall v\in V_h.
\label{picm3}
\end{multline}
Combining the third and fourth terms in \eqref{picm3} using $ab+cd= a(b-d)+(a+c)d$, we can write
\begin{multline}
b^*(\alpha^k_k e_k,\tilde u_k,v) + b^*( \alpha^k_{k-1}(e_k + e_{k-1}),\tilde u_{k-1},v)
 \\ =
b^*(\alpha^k_k e_k,\tilde e_k,v) + b^*( ( \alpha_k + \alpha^k_{k-1}) e_k + \alpha^k_{k-1} e_{k-1}),\tilde u_{k-1},v).
\label{picm4}
\end{multline}
Adding the fifth term from \eqref{picm3} to the third and fourth sum from \eqref{picm4} and proceeding again with  $ab+cd= a(b-d)+(a+c)d$ to combine them  yields
\begin{multline}
b^*(\alpha^k_k e_k,\tilde u_k,v) + b^*( \alpha^k_{k-1}(e_k + e_{k-1}),\tilde u_{k-1},v) + b^*( \alpha^k_{k-2}(e_k + e_{k-1} + e_{k-2}),\tilde u_{k-2},v)
 = \\
b^*(\alpha^k_k e_k,\tilde e_k,v) + b^*( ( \alpha_k + \alpha^k_{k-1}) e_k + \alpha^k_{k-1} e_{k-1}),\tilde e_{k-1},v) \\ + b^*( ( \alpha_k + \alpha^k_{k-1} + \alpha^k_{k-2}) e_k  + (\alpha^k_{k-1} + \alpha^k_{k-2}) e_{k-1}) + \alpha^k_{k-2} e_{k-2},\tilde u_{k-2},v).
\label{picm5}
\end{multline}
Continuing this procedure until all the terms (after the second term) from \eqref{picm3} are added in gives us
\begin{multline}
b^*(\alpha^k_k e_k,\tilde u_k,v) + b^*( \alpha^k_{k-1}(e_k + e_{k-1}),\tilde u_{k-1},v) + b^*( \alpha^k_{k-2}(e_k + e_{k-1} + e_{k-2}),\tilde u_{k-2},v)
\\
+ ... +
b^*( \alpha^k_{k-m} \sum_{j=k-m}^k e_j ,\tilde u_{k-m},v)  
= b^*(\alpha^k_k e_k,\tilde e_k,v) + b^*( ( \alpha^k_k + \alpha^k_{k-1}) e_k + \alpha^k_{k-1} e_{k-1}),\tilde e_{k-1},v) \\ 
+ b^*( ( \alpha^k_k + \alpha^k_{k-1} + \alpha^k_{k-2}) e_k  + (\alpha^k_{k-1} + \alpha^k_{k-2}) e_{k-1}) + \alpha^k_{k-2} e_{k-2},\tilde e_{k-2},v) \\
+ ... + 
b^*\left( \sum_{j=k-m+1}^{k} \alpha^k_j e_k  +  \sum_{j=k-m+1}^{k-1} \alpha^k_j e_{k-1} + \sum_{j=k-m+1}^{k-2} \alpha_j e_{k-2}+ ... + \alpha^k_{k-m+1} e_{k-m+1},\tilde e_{k-m+1},v\right)
\\ + 
b^*\left( \sum_{j=k-m}^k \alpha^k_j e_k  +  \sum_{j=k-m}^{k-1} \alpha^k_j e_{k-1} + \sum_{j=k-m}^{k-2} \alpha^k_j e_{k-2}+ ... + \alpha^k_{k-m} e_{k-m},\tilde u_{k-m},v\right).
\label{picm6}
\end{multline}
%
%
%
Using \eqref{picm6} in \eqref{picm2} gives us for all $v\in V_h$ that
\begin{multline}
\nu(\nabla w_{k+1},\nabla v) + b^*(u_k,w_{k+1},v) 
+ b^*(\alpha^k_k e_k,\tilde e_k,v) + b^*( ( \alpha^k_k + \alpha^k_{k-1}) e_k + \alpha^k_{k-1} e_{k-1}),\tilde e_{k-1},v) \\ 
+ b^*( ( \alpha^k_k + \alpha^k_{k-1} + \alpha^k_{k-2}) e_k  + (\alpha^k_{k-1} + \alpha^k_{k-2}) e_{k-1}) + \alpha^k_{k-2} e_{k-2},\tilde e_{k-2},v) \\
+ ... + 
b^*\left( \sum_{j=k-m+1}^{k} \alpha^k_j e_k  +  \sum_{j=k-m+1}^{k-1} \alpha^k_j e_{k-1} + \sum_{j=k-m+1}^{k-2} \alpha_j e_{k-2}+ ... + \alpha^k_{k-m+1} e_{k-m+1},\tilde e_{k-m+1},v\right)
\\ + 
b^*\left(  e_k  +  \sum_{j=k-m}^{k-1} \alpha^k_j e_{k-1} + \sum_{j=k-m}^{k-2} \alpha^k_j e_{k-2}+ ... + \alpha^k_{k-m} e_{k-m},\tilde u_{k-m},v\right)=0,
\label{picm7}
\end{multline}
where in the last term we used on the $e_k$ coefficient that $\sum_{j=k-m}^k \alpha^k_j=1$.  Choosing $v=w_{k+1}$ vanishes the second left hand side term, and after applying \eqref{bstarbound} to all $b^*$ terms except that last one where we apply Lemma \ref{weidentities}, we get 
\begin{multline}
\nu \| \nabla w_{k+1} \| \le 
M | \alpha_k^k| \| \nabla e_k \| \| \nabla \tilde e_k \| 
+ M ( | \alpha_k^k + \alpha^k_{k-1} |  \| \nabla e_k \| + | \alpha_k^k|  \|\nabla e_{k-1} \|) \| \nabla \tilde e_{k-1} \|  + ... \\
+ M ( | \sum_{j=k-m+1}^{k} \alpha^k_j  |  \| \nabla e_k \| + | \sum_{j=k-m+1}^{k-1} \alpha^k_j |  \| \nabla e_{k-1}\| + ... + | \alpha^k_{k-m+1} | \| \nabla e_{k-m+1} \| ) \| \nabla \tilde e_{k-m+1} \| \\
+ \bigg| b^*\left(  \sum_{j=k-m}^k \alpha_j^k w_j,\tilde u_{k-m},v\right) \bigg|.
\end{multline}
Next, we utilize that $\| \nabla \tilde e_j \| \le \kappa \| \nabla e_{j-1} \|$ (which holds due to Lemma \ref{gproperties}) and apply \eqref{H1L2bound} to the last term on the right hand side to obtain the bound
\begin{multline}
\nu \| \nabla w_{k+1} \| \le 
M \kappa \bigg(  | \alpha_k^k| \| \nabla e_k \| \| \nabla e_{k-1} \| 
+  ( | \alpha_k^k + \alpha^k_{k-1} |  \| \nabla e_k \| + | \alpha_k^k|  \|\nabla e_{k-1} \|) \| \nabla e_{k-2} \| \\
+ ... +  ( | \sum_{j=k-m+1}^{k} \alpha^k_j  |  \| \nabla e_k \| + | \sum_{j=k-m+1}^{k-1} \alpha^k_j |  \| \nabla e_{k-1}\| + ... + | \alpha^k_{k-m+1} | \| \nabla e_{k-m+1} \| ) \| \nabla e_{k-m} \| \bigg)  \\
+ M C_P^{-1/2}   \| \sum_{j=k-m}^k \alpha^k_j w_j \| ^{1/2} \| \sum_{j=k-m}^k \alpha^k_j \nabla w_j \| ^{1/2} \| \nabla \tilde u_{k-m} \|. \label{picm9}
\end{multline}
Because  $\sum_{j=k-m}^k \alpha^k_j=1$, each coefficient that is an absolute value of a sum of $\alpha$'s is bounded by $\frac{(m+3)\bar\alpha }{2}$.  Using this together with the stability of the Picard iteration \eqref{Picbound} that gives $\| \nabla \tilde u_{k-m} \|\le \nu^{-1} \| f \|_{-1}$, multiplying both sides by $\nu^{-1}$ and applying the definition of $\theta_{k,m}^{H^1_0}$ yields the bound
\begin{multline}
\| \nabla w_{k+1} \| \le 
M \nu^{-1} \kappa \frac{(m+3)\bar\alpha}{2}  \bigg(  \|  \nabla e_k \| \| \nabla e_{k-1} \| 
+  (   \| \nabla e_k \| +  \|\nabla e_{k-1} \|) \| \nabla e_{k-2} \| \\
+ ... +  (  \| \nabla e_k \| +   \| \nabla e_{k-1}\| + ... +  \| \nabla e_{k-m+1} \| ) \| \nabla e_{k-m} \| \bigg)  
+ \kappa C_P^{-1/2} {\theta_{k,m}^{H^1_0}}^{1/2}  \| \nabla w_k \|^{1/2} \| \sum_{j=k-m}^k \alpha^k_j w_j \| ^{1/2}. \label{picm10}
\end{multline}
Grouping terms and multiplying the last right hand side term by 
\[
1 = \frac{ \| \sum_{j=k-m}^k \hat \alpha^k_j w_j \| ^{1/2}}{ \| \sum_{j=k-m}^k \hat \alpha^k_j w_j \| ^{1/2}}
= \frac{ {\theta_{k,m}^{L^2}}^{1/2} \| w_k \| ^{1/2}}{ \| \sum_{j=k-m}^k \hat \alpha^k_j w_j \| ^{1/2}}
\]
gives us that
\begin{multline}
\| \nabla w_{k+1} \| \le
\kappa {\theta_{k,m}^{H^1_0}}^{1/2} {\theta_{k,m}^{L^2}}^{1/2} \left(  \frac{ \| \sum_{j=k-m}^k \alpha^k_j w_j \| }{ \| \sum_{j=k-m}^k \hat \alpha^k_j w_j \|  }  \frac{ \| w_k \| }{ C_P \| \nabla w_k\|}\right)^{1/2} \| \nabla w_k \|\\
+
M \nu^{-1} \kappa \frac{m(m+3)\bar\alpha}{2} \sum_{j=k-m+1}^k \sum_{i=k-m}^{j-1} \| \nabla e_i \| \| \nabla e_j \|.
  \label{picm11}
\end{multline}
Finally, applying \eqref{ewbound} produces
\begin{multline}
\| \nabla w_{k+1} \| \le
\kappa {\theta_{k,m}^{H^1_0}}^{1/2} {\theta_{k,m}^{L^2}}^{1/2} \left(  \frac{ \| \sum_{j=k-m}^k \alpha^k_j w_j \| }{ \| \sum_{j=k-m}^k \hat \alpha^k_j w_j \|  }  \frac{ \| w_k \| }{ C_P \| \nabla w_k\|}\right)^{1/2} \| \nabla w_k \|\\
+
K^2 M \nu^{-1} \kappa \frac{m(m+3)\bar\alpha}{2}\sum_{j=k-m+1}^k \sum_{i=k-m}^{j-1} \| \nabla w_j  \| \| \nabla w_j \|,
  \label{picm12}
\end{multline}
which completes the proof.
\end{proof}

\subsection{The case of AA optimization in the $L^2$ norm}

We now consider convergence of AA-Picard in the case when the $L^2$ norm is used for the AA optimization.  This causes two key differences compared to the $H^1_0$ optimization case studied above.  First, $\theta_{k,m}^{L^2}$ defined in \eqref{thetaL2} is now the natural gain of the optimization problem.  However, due to the way the nonlinearity was analyzed in the $H^1_0$ optimization norm case as well as the $L^2$ optimization norm case below, both $\theta_{k,m}^{L^2}$ and $\theta_{k,m}^{H^1_0}$ appear as coefficients of the first order term in both of the convergence estimates.  

The second difference is that \eqref{ewbound} is no longer implied by Assumption \ref{assume}.  Instead, such an implication would hold only in the $L^2$ norm.  Hence we reformulate in the following way, beginning from Assumption \ref{assume} and rewriting it as
\[
\tilde \sigma \| e_j \| \le \| w_{j+1}-w_j \| \mbox{  where } \tilde \sigma = \sigma \min_j \bigg\{  \frac{ \| \nabla e_j\|}{\| e_j \|} \frac{ \| w_{j+1} - w_j  \|}{ \| \nabla (w_{j+1}-w_j) \| } \bigg\}.
\]
While we expect $\sigma$ and $\tilde \sigma$ to be similar due to the average of the ratios of a $H^1_0$ to $L^2$ norm for a term multiplied by the ratio of a $L^2$ to $H^1_0$ norm of a different term to be near 1, we cannot rule out extreme cases where $\sigma$ and $\tilde \sigma$ vary significantly.  Hence we assume that $\tilde \sigma>0$.  
Note that $\tilde \sigma$ only affects the higher order terms in the convergence estimate below.
\begin{equation}
 \| \nabla e_j \| \le \tilde K \| \nabla w_{j} \|, \label{ewbound2}
\end{equation}
where $\tilde K$ depends on $\tilde \sigma$, $\theta_k^{L^2}$, and the degree to which the elements of $\{ (w_{k} - w_{k-1}), ... , (w_{k-m}-w_{k-m-1}) \}$ are linearly independent with respect to the $L^2$ inner product \cite{PR21}.

\begin{theorem}[Convergence result for $L^2$ norm in AA optimization] 
\label{thm:m1L2}
Under Assumption \ref{assume} and assuming $\tilde\sigma>0$, $0<|\hat \alpha_j^k|<\bar{\hat \alpha} \ \forall k,j$ for some fixed $\bar{\hat \alpha}$, and that the residual differences $\{ w_j - w_{j-1} \}_{j=1,2,...,k}$ are not linearly dependent in the $L^2$ inner product, the depth $m$ AA-Picard iterates found using $L^2$ as the optimization norm satisfy
\begin{multline}
\| \nabla w_{k+1} \| \le
\kappa {\theta_{k,m}^{L^2}}^{1/2}  {\theta_{k,m}^{H^1_0}}^{1/2}  \left(  \frac{ \| \sum_{j=k-m}^k \hat \alpha^k_j \nabla w_j \| }{ \| \sum_{j=k-m}^k \alpha^k_j \nabla w_j \|  } \right)^{1/2} \left( \frac{ \| w_k \| }{ C_P \| \nabla w_k\|}\right)^{1/2} \| \nabla w_k \|\\
+
\tilde K^2 M \nu^{-1} \kappa \frac{m(m+3)\bar{\hat \alpha}}{2} \sum_{j=k-m+1}^k \sum_{i=k-m}^{j-1} \| \nabla w_i \| \| \nabla w_j \|.
  \label{L22}
\end{multline}
\end{theorem}
\begin{remark}
The convergence estimates for AA-Picard for the case of $L^2$ and $H^1_0$ norms are very similar.  For the first order terms, the difference is a single term difference: a scaling by $ \left(  \frac{ \| \sum_{j=k-m}^k \hat \alpha^k_j \nabla w_j \| }{ \| \sum_{j=k-m}^k \alpha^k_j \nabla w_j \|  } \right)^{1/2}$ (the ratio loss in the $H^1_0$ norm of using $L^2$ optimal coefficients instead of $H^1_0$ coefficients in the objective function) for the $L^2$ case  versus  $ \left(  \frac{ \| \sum_{j=k-m}^k  \alpha^k_j  w_j \| }{ \| \sum_{j=k-m}^k \hat \alpha^k_j w_j \|  } \right)^{1/2}$ (the ratio loss in the $L^2$ norm of using $H^1_0$ optimal coefficients instead of $L^2$ coefficients in the objective function).  Our numerical tests show these two quantities are on average quite close.

The difference in the higher order terms between the $H^1_0$ and $L^2$  optimization norm cases are the constants $K$ versus $\tilde K$ and $\bar \alpha$ versus $\bar{\hat \alpha}$, respectively.  
\end{remark}

\begin{proof}
The proof begins the exact same way as the proof of Theorem \ref{thmH10m}, with the exception that $\hat \alpha$'s are used instead of $\alpha$'s, since the $\hat \alpha$'s are the coefficients from the $L^2$ optimization.  Since $\sum_{j=k-m}^k \hat \alpha_j=1$, we can trace the proof of Theorem \ref{thmH10m} all the way down to  \eqref{picm9} to obtain
\begin{multline}
\nu \| \nabla w_{k+1} \| \le 
M \kappa \bigg(  | \hat\alpha_k^k| \| \nabla e_k \| \| \nabla e_{k-1} \| 
+  ( | \hat\alpha_k^k + \hat\alpha^k_{k-1} |  \| \nabla e_k \| + | \hat\alpha_k^k|  \|\nabla e_{k-1} \|) \| \nabla e_{k-2} \| \\
+ ... +  ( | \sum_{j=k-m+1}^{k} \hat\alpha^k_j  |  \| \nabla e_k \| + | \sum_{j=k-m+1}^{k-1} \hat\alpha^k_j |  \| \nabla e_{k-1}\| + ... + | \hat\alpha^k_{k-m+1} | \| \nabla e_{k-m+1} \| ) \| \nabla e_{k-m} \| \bigg)  \\
+ M C_P^{-1/2}   \| \sum_{j=k-m}^k \hat\alpha^k_j w_j \| ^{1/2} \| \sum_{j=k-m}^k \hat\alpha^k_j \nabla w_j \| ^{1/2} \| \nabla \tilde u_{k-m} \|. \label{L21}
\end{multline}
For the last term in \eqref{L21}, we first use the definition of $\theta_{k,m}^{L^2}$ and that $\| \nabla u_{k-m} \|\le \nu^{-1} \| f\|_{-1}$ by Lemma \ref{gproperties} to get
\begin{align}
M C_P^{-1/2}  & \| \sum_{j=k-m}^k \hat\alpha^k_j w_j \| ^{1/2} \| \sum_{j=k-m}^k \hat\alpha^k_j \nabla w_j \| ^{1/2} \| \nabla \tilde u_{k-m} \| \nonumber \\
& \le \kappa \nu C_P^{-1/2} \theta_{k,m}^{L^2} \|  w_k \| ^{1/2} \| \sum_{j=k-m}^k \hat\alpha^k_j \nabla w_j \| ^{1/2} \nonumber \\
& \le \kappa \nu C_P^{-1/2} \theta_{k,m}^{L^2} \|  w_k \| ^{1/2}  \frac{ \| \sum_{j=k-m}^k \hat\alpha^k_j \nabla w_j \|}{ \| \sum_{j=k-m}^k \alpha^k_j \nabla w_j \|}^{1/2}  \| \sum_{j=k-m}^k \alpha^k_j \nabla w_j \|^{1/2},
\end{align}
where on the last step we multiplied by $1= \frac{ \| \sum_{j=k-m}^k \alpha^k_j \nabla w_j \|}{ \| \sum_{j=k-m}^k \alpha^k_j \nabla w_j \|}^{1/2}.$  Now using the  definition of $\theta_{k,m}^{H^1_0}$ and rearranging terms, we have the bound
\begin{multline}
M C_P^{-1/2}   \| \sum_{j=k-m}^k \hat\alpha^k_j w_j \| ^{1/2} \| \sum_{j=k-m}^k \hat\alpha^k_j \nabla w_j \| ^{1/2} \| \nabla \tilde u_{k-m} \|  
\le \\
\kappa \nu  \theta_{k,m}^{L^2} \theta_{k,m}^{H^1_0} \| \nabla w_k \|  \frac{\|  w_k \|}{C_P \| \nabla w_k \|}^{1/2}  \frac{ \| \sum_{j=k-m}^k \hat\alpha^k_j \nabla w_j \|}{ \| \sum_{j=k-m}^k \alpha^k_j \nabla w_j \|}^{1/2}.  
\end{multline}

For the remaining right hand side terms of \eqref{L21}, we proceed just as in the proof of Theorem \ref{thmH10m} to bound these terms, but using $\tilde K$ from \eqref{ewbound2} instead of $K$ from \eqref{ewbound}, and $\bar{ \hat \alpha}$ as a bound on the max absolute value of the $\hat \alpha^k_j$'s.  With these bounds, we reduce \eqref{L21} to 
\begin{multline}
\| \nabla w_{k+1} \| \le
\kappa {\theta_{k,m}^{L^2}}^{1/2}  {\theta_{k,m}^{H^1_0}}^{1/2}  \left(  \frac{ \| \sum_{j=k-m}^k \hat \alpha^k_j \nabla w_j \| }{ \| \sum_{j=k-m}^k \alpha^k_j \nabla w_j \|  }  \frac{ \| w_k \| }{ C_P \| \nabla w_k\|}\right)^{1/2} \| \nabla w_k \|\\
+
\tilde K^2 M \nu^{-1} \kappa \frac{m(m+3)\bar{\hat \alpha}}{2} \sum_{j=k-m+1}^k \sum_{i=k-m}^{j-1} \| \nabla w_i \| \| \nabla w_j \|.
  \label{L22a}
\end{multline}
This completes the proof.
\end{proof}

\section{Numerical Tests}

We give results in this section for three numerical tests that illustrate our theory and compare the various methods.  For each test, we compare convergence of AA-Picard for NSE with both $H^1_0$ and $L^2$ norms used for the AA optimization problem, which our theory shows will have similar convergence behavior even though $H^1_0$ is the norm of the associated Hilbert space of the iteration.  Additionally, we compare AA-Picard results for the commonly used  $\ell^2$ norm and the  diagonally lumped $L^2$ norm (same as $L^2$ norm except using a diagonally lumped mass matrix in the norm calculation).

We test the proposed methods on the 2D channel flow past a cylinder, 2D driven cavity, and the 3D driven cavity, each on three successively refined meshes.  For all problems, the initial guess is $u_0=0$ in the interior but satisfying boundary conditions, and no continuation methods are used in any of our tests.  All tests utilize pointwise divergence-free Scott-Vogelius mixed finite elements.  Our tests use varying $m$ and $Re$, and when $m=0$ we refer to AA-Picard simply as Picard.  The stopping criteria is when the nonlinear residual falls below $10^{-8}$ in the $H^1_0$ norm.  Results from \cite{PR21,EPRX20} show that for AA-Picard for NSE, larger $m$ is more effective as $Re$ gets larger.

\subsection{2D channel flow past a cylinder}

\begin{figure}[ht]
\center
\includegraphics[width = .85\textwidth, height=.17\textwidth,viewport=170 20 1230 250, clip]{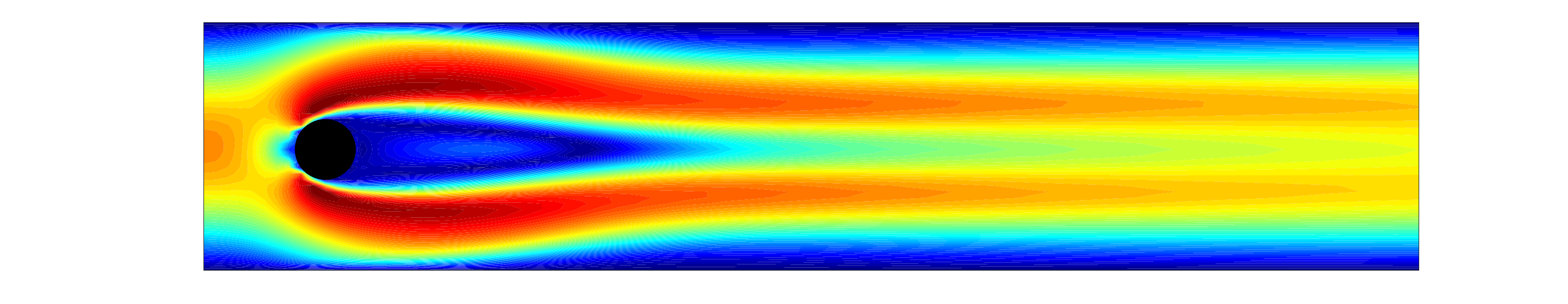}\\
\includegraphics[width = .85\textwidth, height=.17\textwidth,viewport=170 20 1230 250, clip]{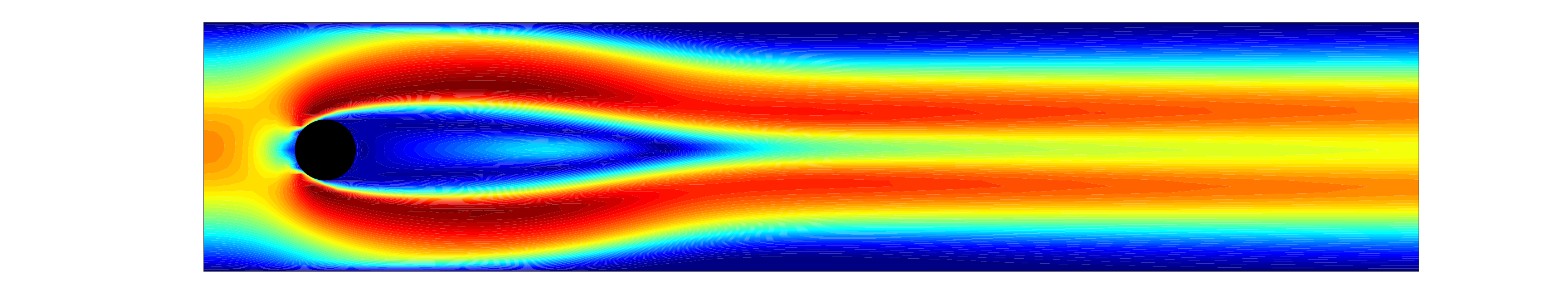}\\
\caption{\label{cylplot} Shown above are steady solution speed contours for the 2D flow past a cylinder at $Re$=100 (top) and 200 (bottom).}
\end{figure}

For our first experiment, we test AA-Picard with different choices of AA optimization norm on channel flow past a cylinder. Many studies on this benchmark test can be found in the literature, e.g. \cite{ST96,J04}, and it is an interesting test because as the Reynolds number increases the physics of the problem get more complex and nonlinear solvers have more trouble converging.  The domain is a $2.2\times0.41$ rectangular channel, with a cylinder of diameter  $0.1$ centered at $(0.2, 0.2)$ (with the origin set as the bottom left corner of the rectangle).

\begin{figure}[ht]
\center
\includegraphics[width = .85\textwidth, height=.17\textwidth,viewport=100 160 500 240, clip]{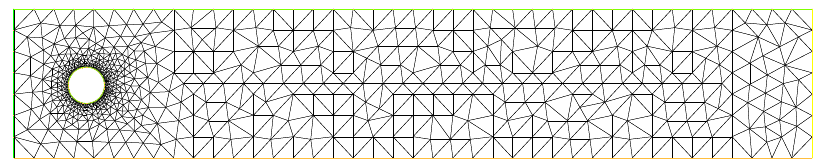}\\
\includegraphics[width = .85\textwidth, height=.17\textwidth,viewport=100 160 500 240, clip]{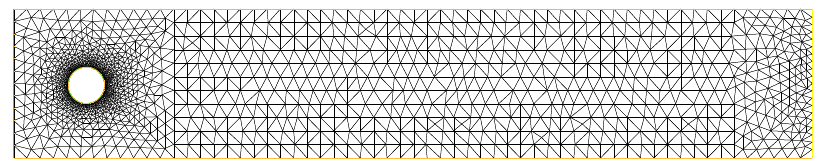}\\
\includegraphics[width = .85\textwidth, height=.17\textwidth,viewport=100 160 500 240, clip]{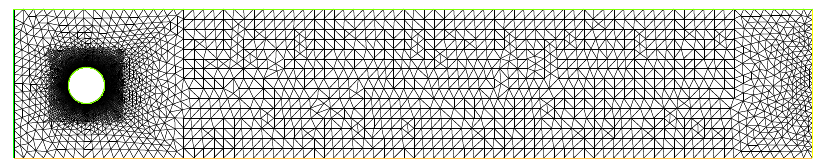}\\
\caption{\label{cylmesh} Shown above are the meshes used for the 2D flow past a cylinder, before the barycenter refinement is applied.}
\end{figure}


\begin{figure}[ht]
\center
\includegraphics[width = .32\textwidth, height=.32\textwidth,viewport=0 0 550 430, clip]{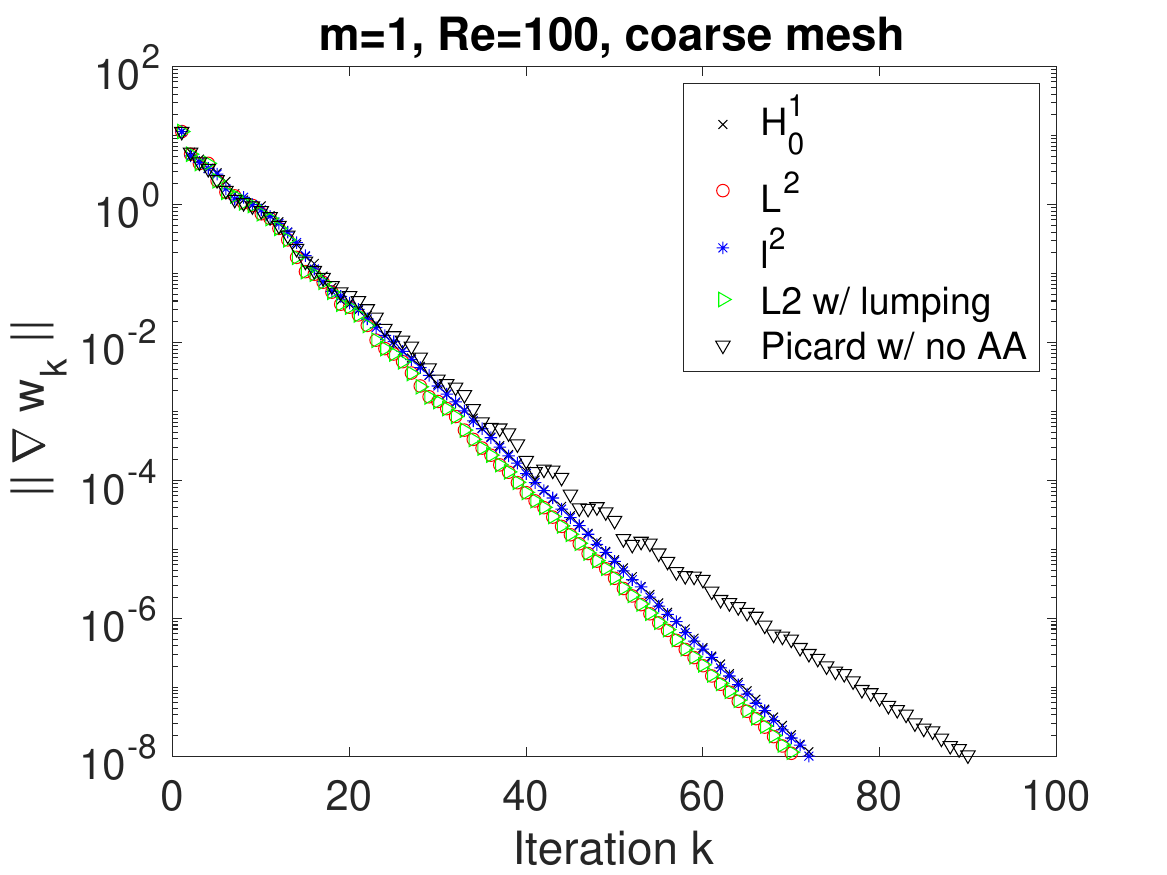}
\includegraphics[width = .32\textwidth, height=.32\textwidth,viewport=0 0 550 430, clip]{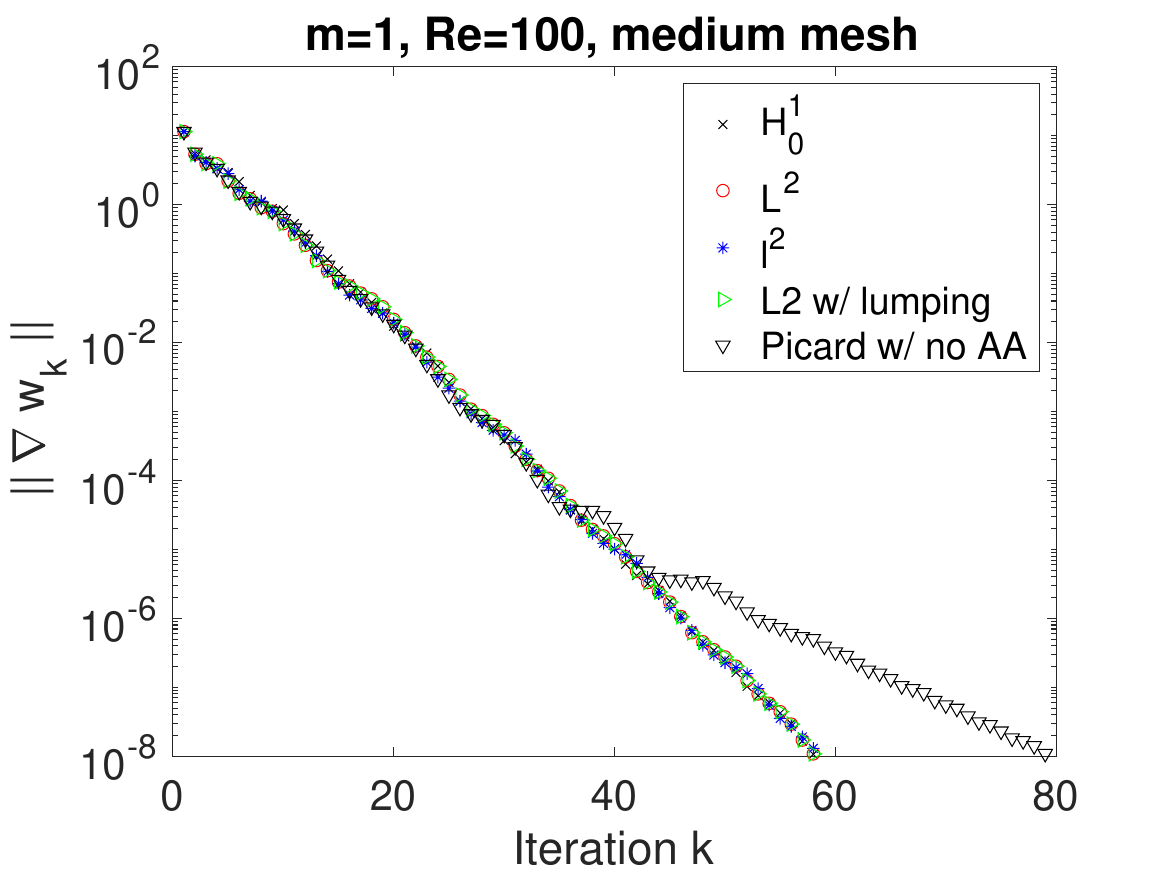}
\includegraphics[width = .32\textwidth, height=.32\textwidth,viewport=0 0 550 430, clip]{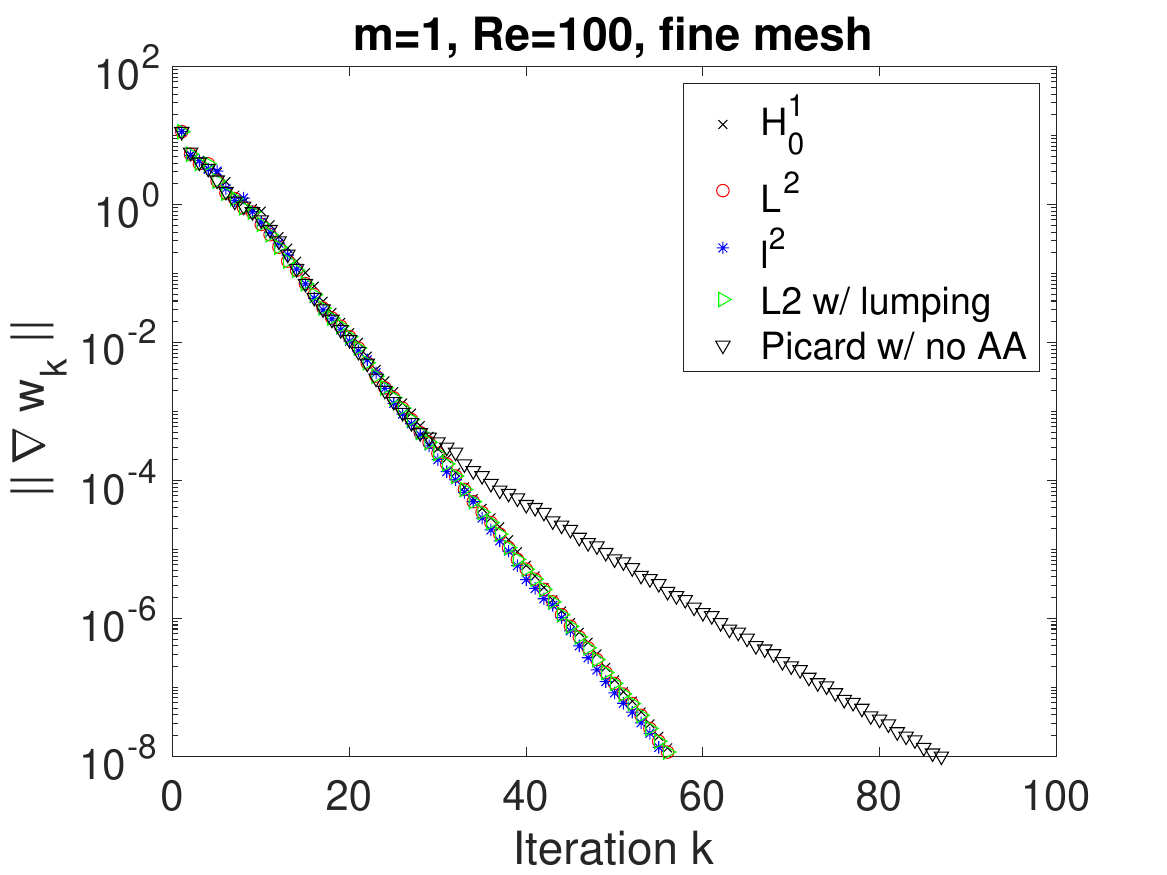}\\
\includegraphics[width = .32\textwidth, height=.32\textwidth,viewport=0 0 550 430, clip]{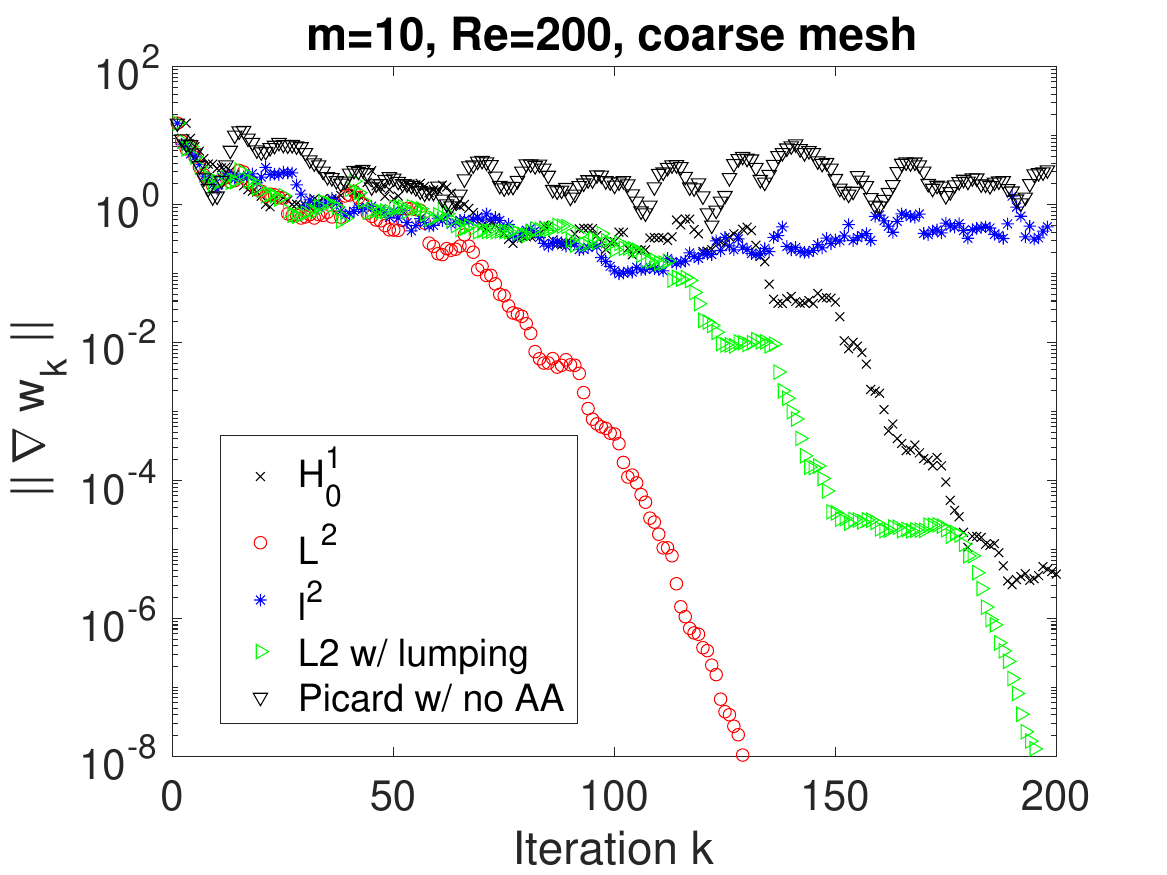}
\includegraphics[width = .32\textwidth, height=.32\textwidth,viewport=0 0 550 430, clip]{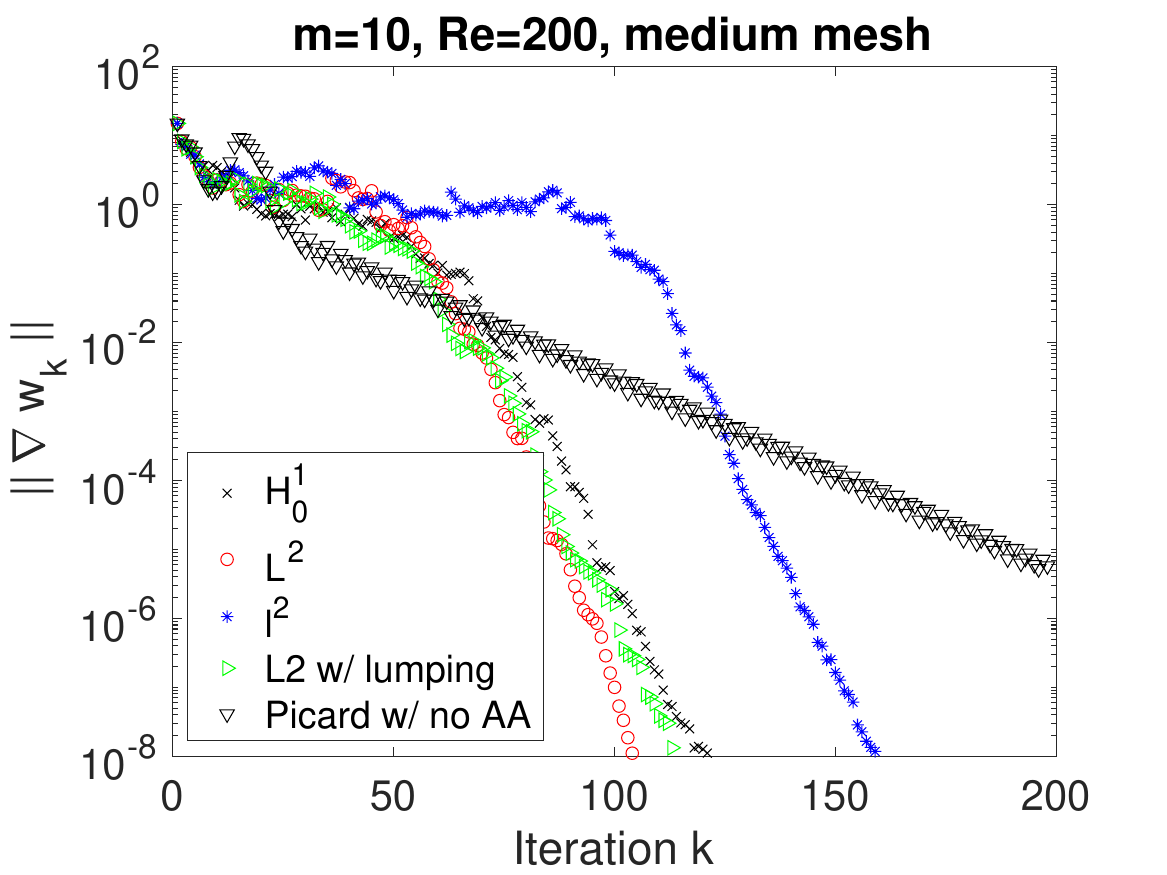}
\includegraphics[width = .32\textwidth, height=.32\textwidth,viewport=0 0 550 430, clip]{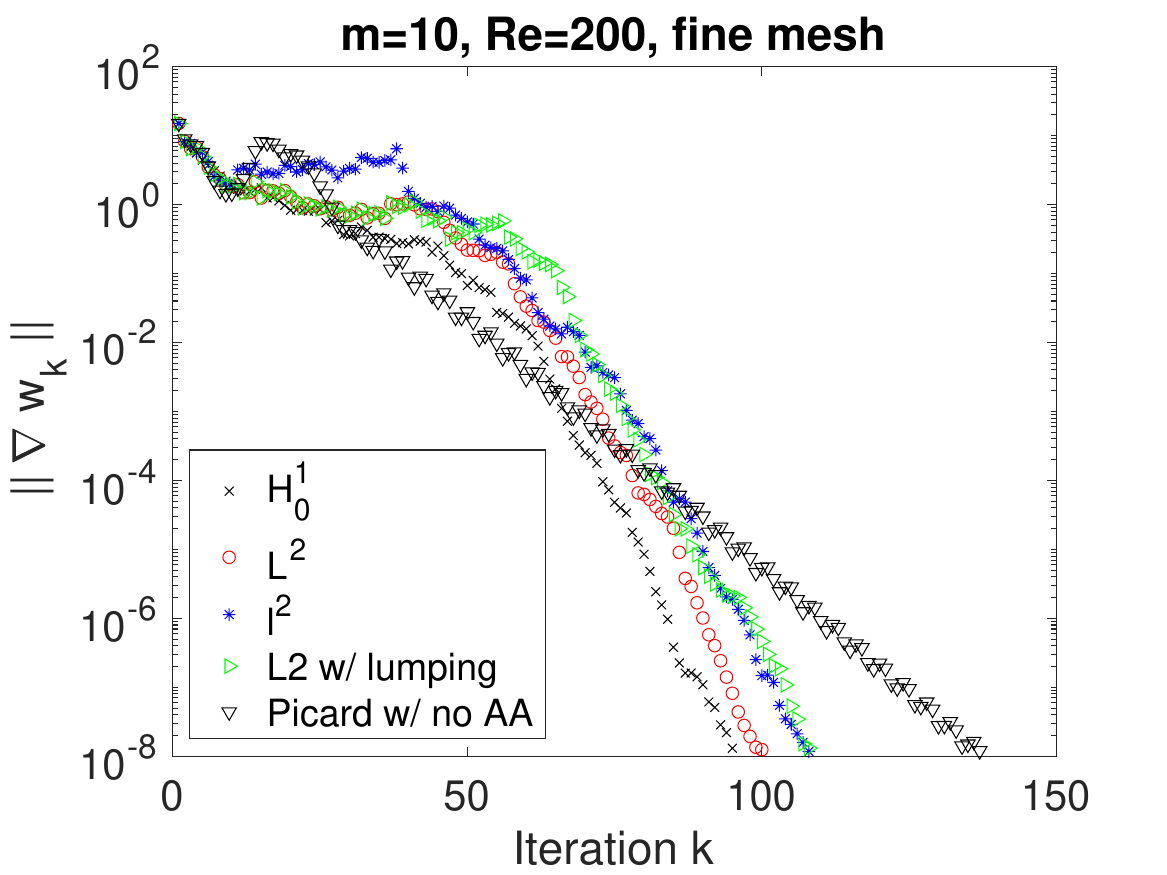}
\caption{\label{cylconv2} Shown above are the convergence results for Re=200 channel flow past a cylinder, for varying meshes and with (top) $m=1$ and $Re$=100 and (bottom) $m=10$ and $Re$=200.}
\end{figure}

\begin{figure}[ht]
\center
\includegraphics[width = .46\textwidth, height=.32\textwidth,viewport=0 0 800 460, clip]{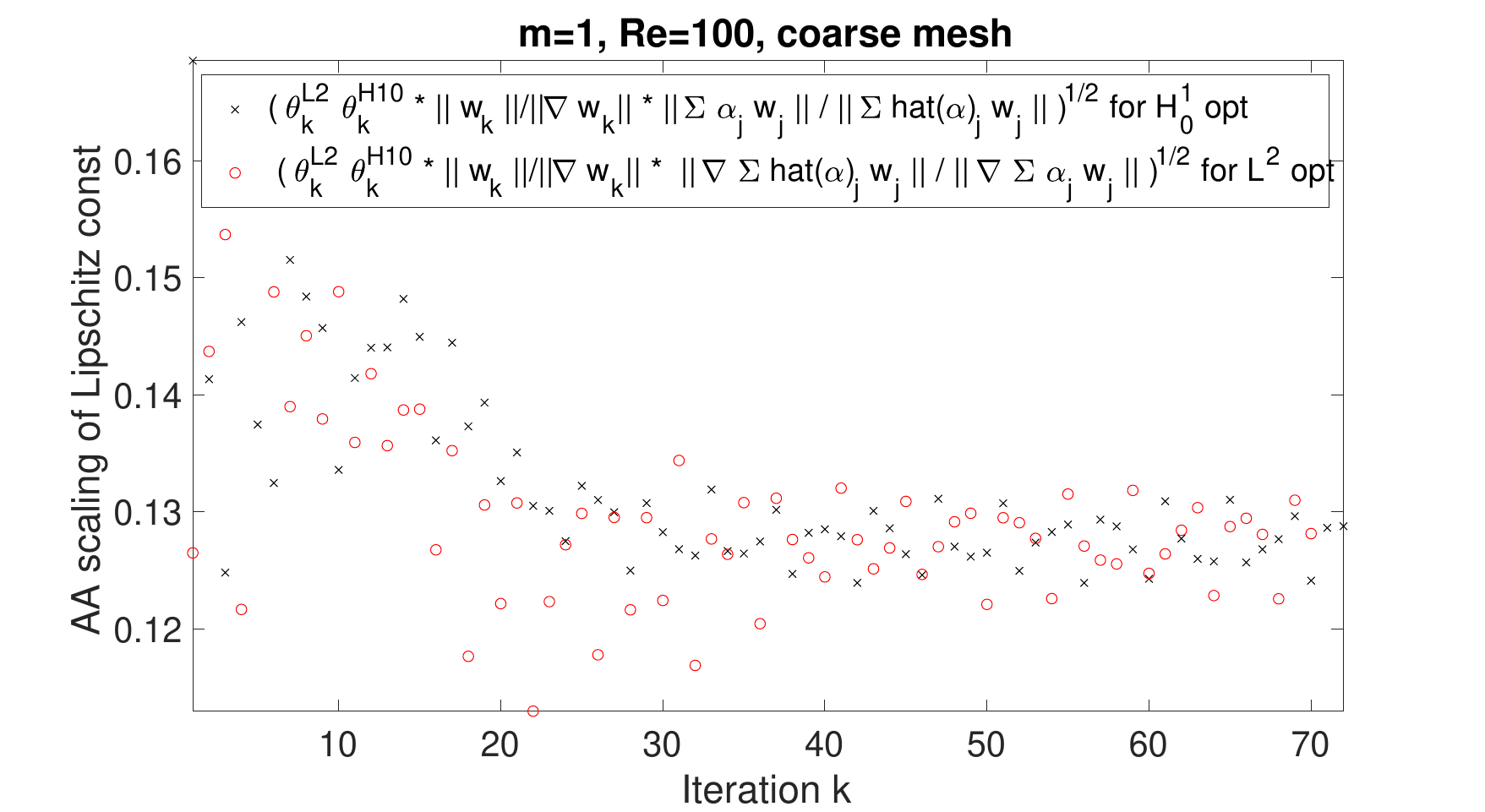}
\includegraphics[width = .46\textwidth, height=.32\textwidth,viewport=0 0 800 460, clip]{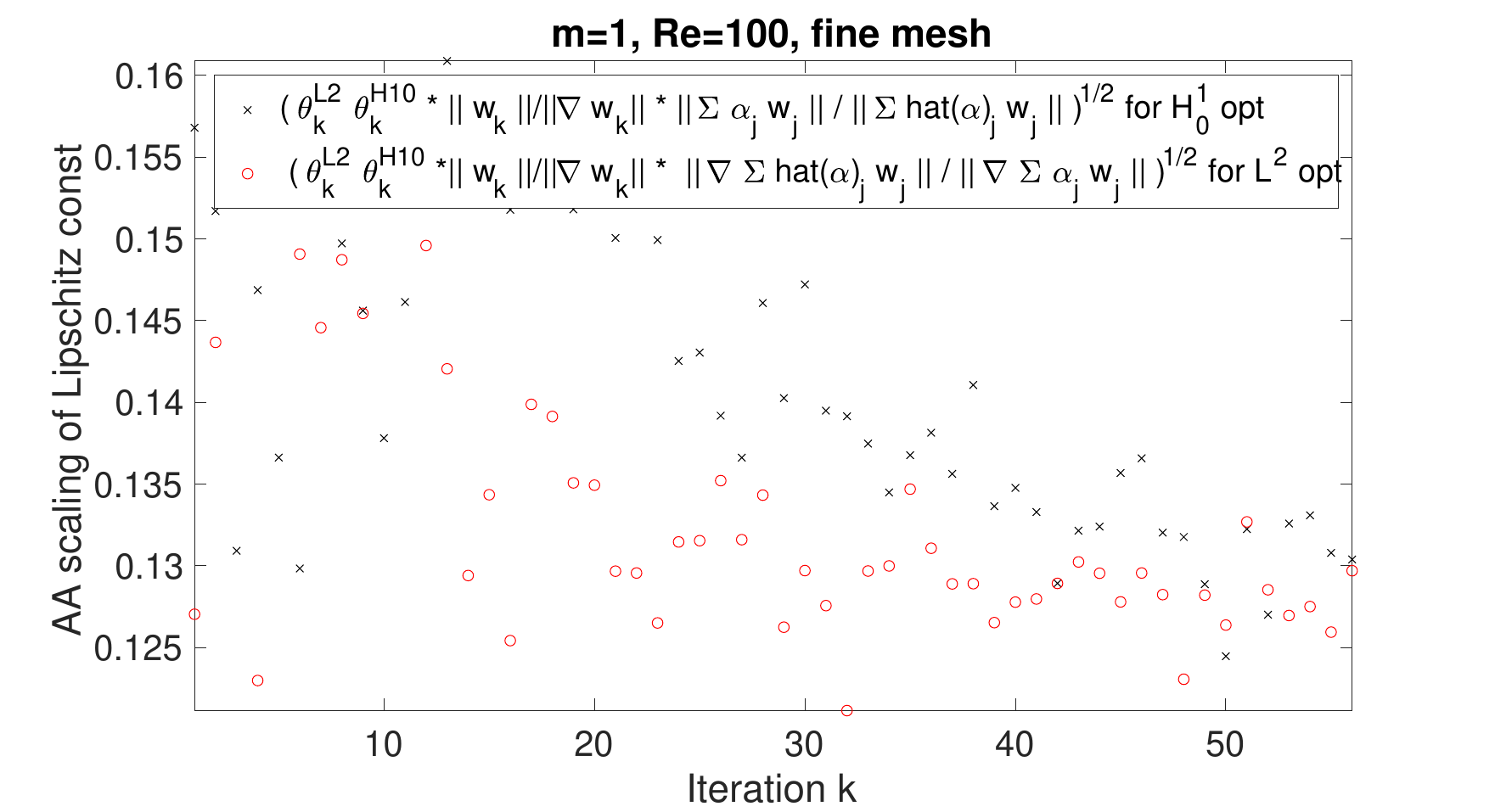}\\
\includegraphics[width = .46\textwidth, height=.32\textwidth,viewport=0 0 800 460, clip]{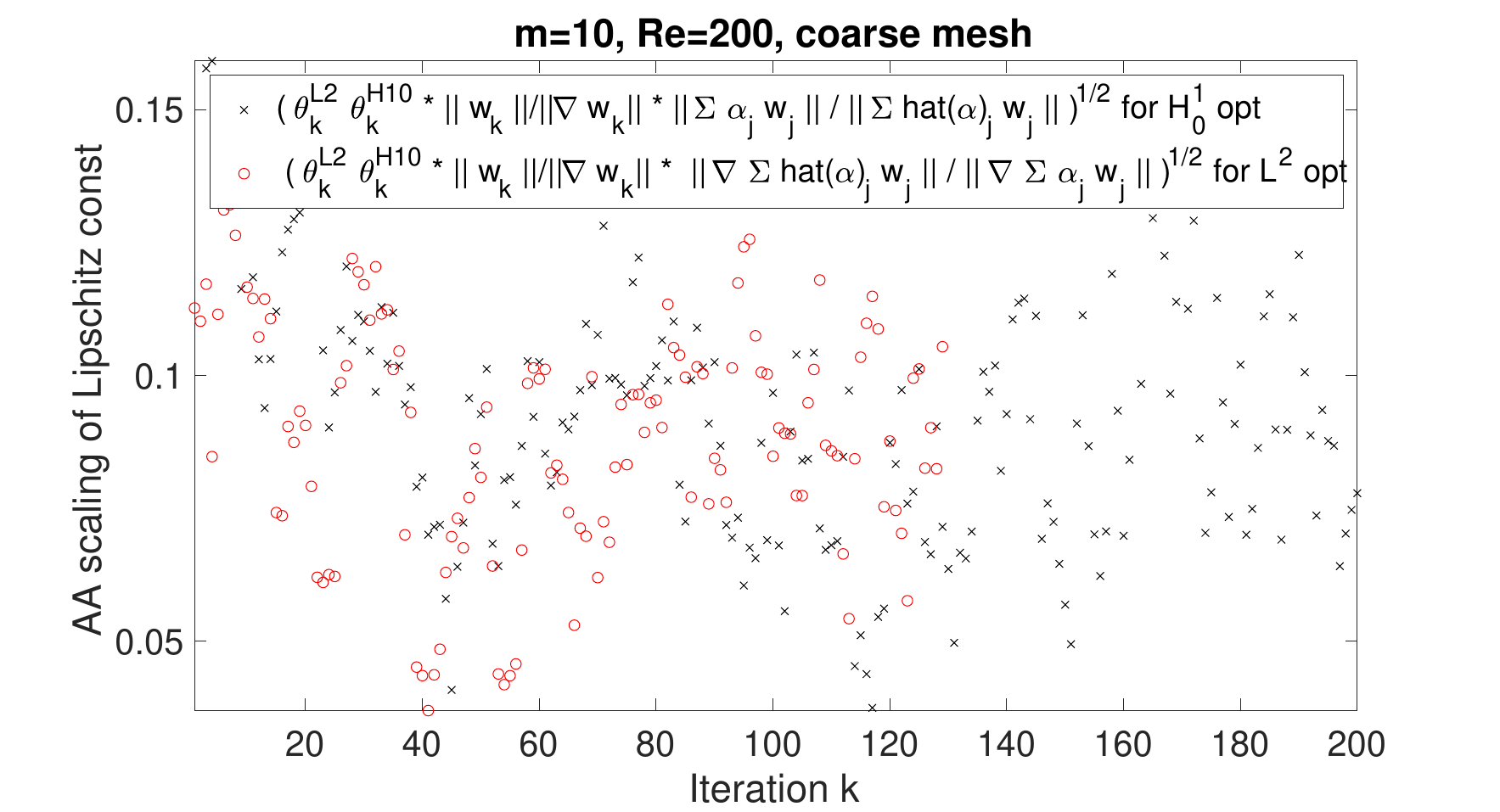}
\includegraphics[width = .46\textwidth, height=.32\textwidth,viewport=0 0 800 460, clip]{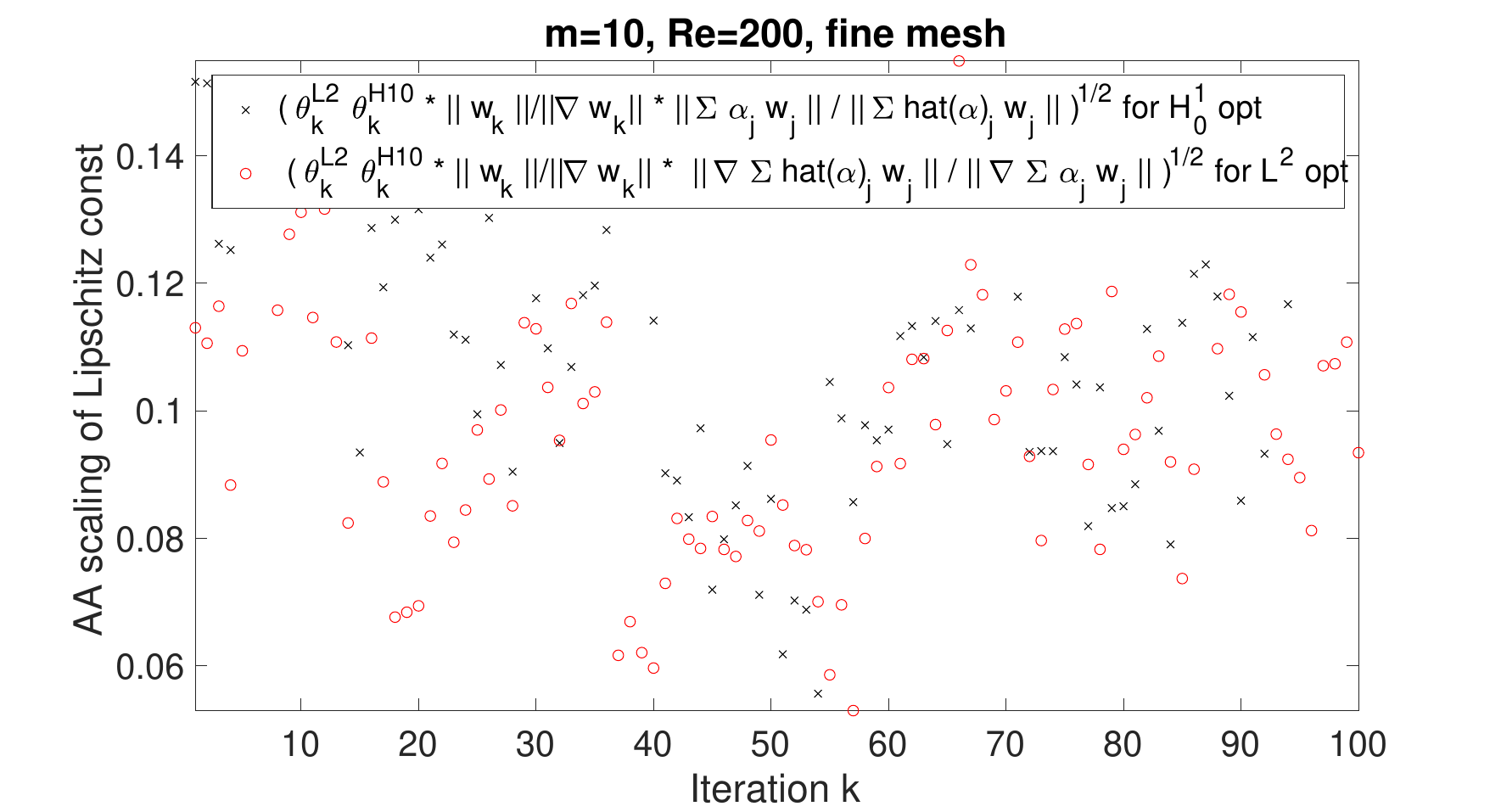}
\caption{\label{cylconv3} Shown above are first order residual factors (i.e. that scale the Lipschitz constant of Picard) for channel flow past a cylinder computations of AA-Picard with $H^1_0$ and $L^2$ norms used for AA optimization.}
\end{figure}

No-slip velocity boundary conditions are imposed on the walls and on the cylinder, and at the inflow the profile is set as 
\begin{align}
	u_1(0,y,t)=&u_1(2.2,y,t)=\frac{6}{0.41^2}y(0.41-y),\nonumber\\
	u_2(0,y,t)=&u_2(2.2,y,t)=0.
\end{align}
The outflow weakly enforces a zero traction (i.e. do nothing) Neumann boundary condition.  This problem has no external forcing, $f=0$. 

We compute with Reynolds numbers $Re$=100 and 200, which translates to $\nu=\frac{1}{1000}$ and $\frac{1}{2000}$ using the length scale as the diameter of the cylinder.  We note that many numerical tests have been run using $Re=100$  \cite{ST96,J04}, and it is known that it admits periodic-in-time solutions for these $Re$.  In fact, the critical Reynolds number $Re_c$ for the shift from steady to time dependent on this problem is near $Re_c=48$ \cite{OR25}.  Still, we search for (and find) steady NSE solutions at these Reynolds numbers (recall that the steady NSE admits solutions for any $Re$, although solutions may not be unique \cite{Laytonbook}).  Solution plots for these $Re$ are shown in Figure \ref{cylplot}.

We compute on three successively refined meshes, each of which is more refined around the cylinder than downstream in the channel.  The meshes are created by Delaunay triangulations that are then further refined with a barycenter refinement, except the finest mesh is additionally refined around the cylinder before the barycenter refinement.  The node spacing around the cylinder on the finest mesh is 9.24e-4.  The meshes are shown before the barycenter refinement is applied in Figure \ref{cylmesh}; we will refer to them as the coarse, medium and fine meshes.  We compute using $(P_2,P_1^{disc})$ Scott-Vogelius elements, which on these meshes provides for 29K, 68K, and 297K total degrees of freedom (dof) for the coarse, medium and fine meshes, respectively.  We test using $m=1$ for $Re$=100 and $m=10$ for $Re$=200, using $H^1_0$, $L^2$, $\ell^2$ and diagonally lumped $L^2$ norms for the AA optimization.  

Convergence results are shown in Figure \ref{cylconv2}.  For $Re$=100, we observe that convergence results using the four choices of optimization norms
 all display very similar convergence, and all provide an advantage of Picard.  The convergence of all AA-Picard tests get slightly better as the mesh is refined.  

The situation is more complicated for the tests using $Re$=200 and $m=10$.  First, we note that all methods improve as the meshes get finer.  In fact, on the coarse mesh, Picard and AA-Picard using the $\ell^2$ norm never get small residuals in 200 iterations but both converge on the fine mesh.  On the finest mesh, all AA-Picard tests converge in around 100 iterations, with $H^1_0$ performing the best and $\ell^2$ and diagonally lumped $L^2$ performing the worst, taking about 10 more iterations to converge.  On the medium mesh, $\ell^2$ performs significantly worse than the other choices of optimization norms, which perform roughly the same as each other.  The coarse mesh convergence behavior is very erratic, with $L^2$ performing best, $H^1_0$ and diagonally lumped $L^2$ next best, but with $\ell^2$ not reducing the fixed point residual below $10^{-1}$ at all in 200 iterations.  As the mesh is refined, we see only improved convergence in each method.

Additionally, we plot the first order coefficient terms in the residual expansion from Theorems \ref{thm:m1b} and \ref{thm:m1L2} for $H^1_0$ and $L^2$ optimization norms, respectively.  For the $H^1_0$ norm AA optimization, the first order term coefficient is
\begin{equation}
\kappa (\theta_k^{H^1_0} \theta_k^{L^2})^{1/2} 
\left( \frac{ \| \sum_{j=k-m}^{k}  \alpha_{j}^k w_{j} \|}{\| \sum_{j=k-m}^{k} \hat \alpha_{j}^k w_{j}  \|} \right)^{1/2} 
\left( \frac{ \| w_{k-1} \| }{C_P \| \nabla w_{k-1} \|} \right)^{1/2}, \label{factor1}
\end{equation}
and for the case of the $L^2$ norm, the first order term coefficient differs only in the one factor:
\begin{equation}
\kappa (\theta_k^{H^1_0} \theta_k^{L^2})^{1/2} 
\left( \frac{ \| \nabla \sum_{j=k-m}^{k}  \hat \alpha_{j}^k w_{j} \|}{\| \nabla \sum_{j=k-m}^{k} \alpha_{j}^k w_{j}  \|} \right)^{1/2} 
\left( \frac{ \| w_{k-1} \| }{C_P \| \nabla w_{k-1} \|} \right)^{1/2}, \label{factor2}
\end{equation}

Figure \ref{cylconv3} shows plots of these factors \eqref{factor1}-\eqref{factor2} versus $k$, with the slight modification that $C_P$ and $\kappa$ are omitted from the calculation.  However, these are constants that are fixed with respect to $k$, and so omitting them does not change the comparison of factors.  We observe from the plots that although the factors for $L^2$ optimization norm are slightly smaller in each case, the factors are still fairly close (note the scale on the y-axis does not significantly vary).  From the four plots, the biggest difference in factor averages is when $m=1$ and $Re=100$ when it is $0.109$ for $L^2$ versus $0.097$ for $H^1_0$, and the smallest difference is $0.131$ for $L^2$ versus $0.130$ for $H^1_0$.  This is typical for all of our tests: the factors for $L^2$ tend to be slightly better (but not always), but not by a significant amount.

\subsection{2D driven cavity}

\begin{figure}
\includegraphics[trim = 110pt 40pt 90pt 5pt,clip = true,width=.4\textwidth,height=.4\textwidth]{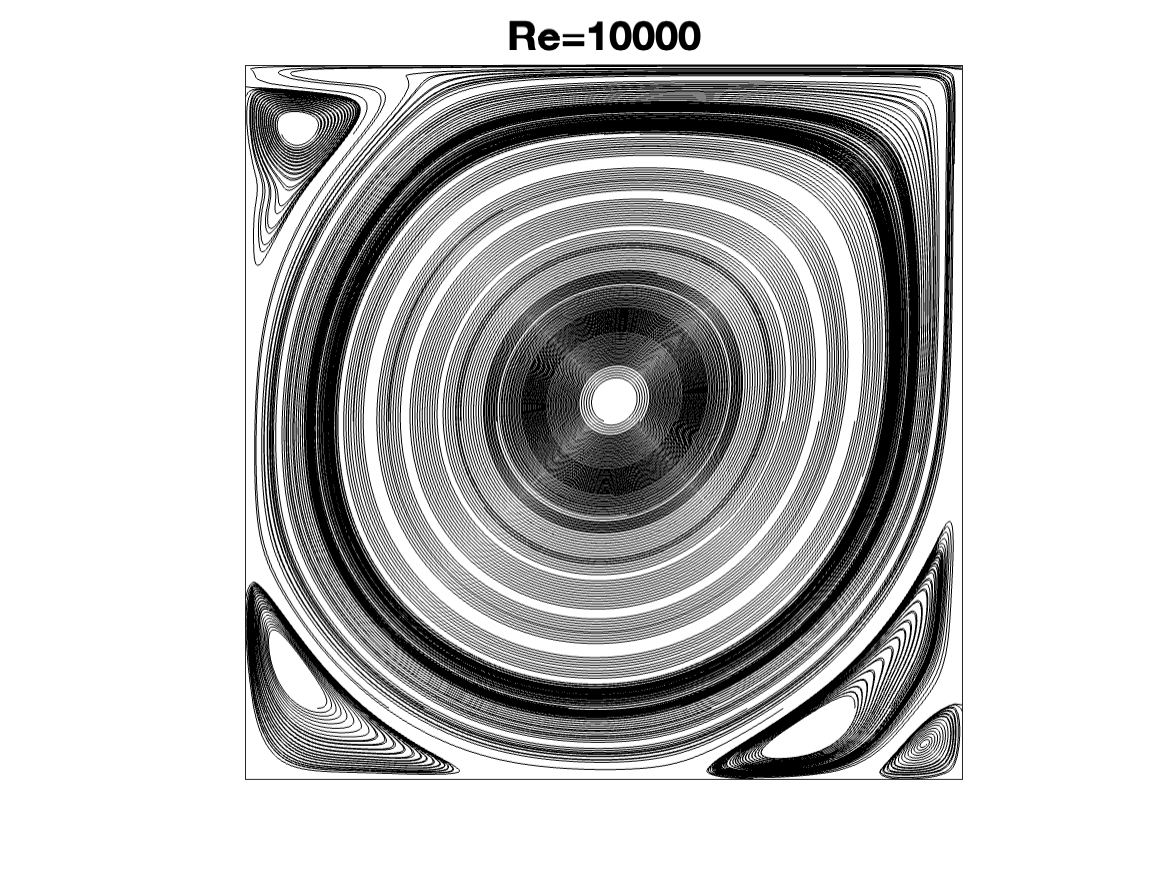}
\caption{\label{cavityplots} Shown above are streamlines of solutions found for the 2D driven cavity problems at $Re$=10000.}
\end{figure}

\begin{figure}
\includegraphics[trim = 0pt 0pt 0pt 0pt,clip = true,width=.32\textwidth,height=.32\textwidth]{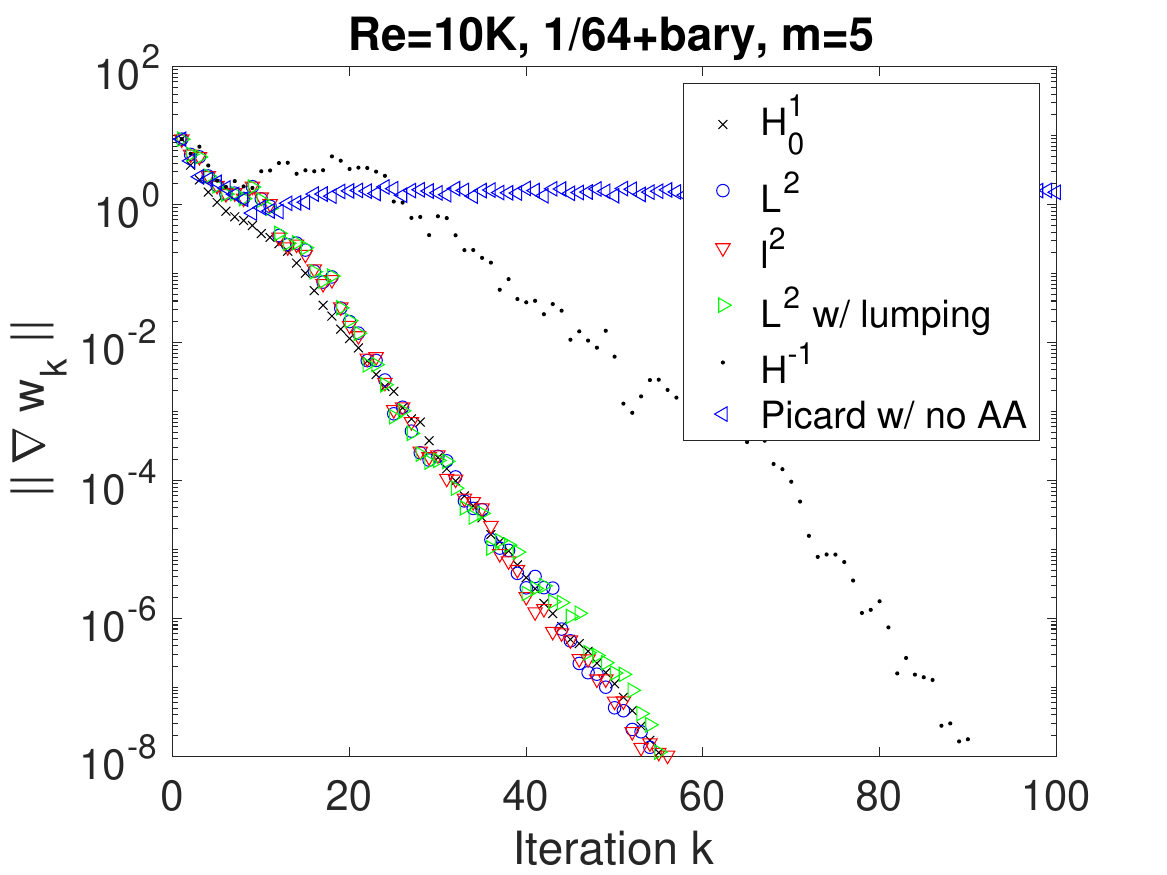}
\includegraphics[trim = 0pt 0pt 0pt 0pt,clip = true,width=.32\textwidth,height=.32\textwidth]{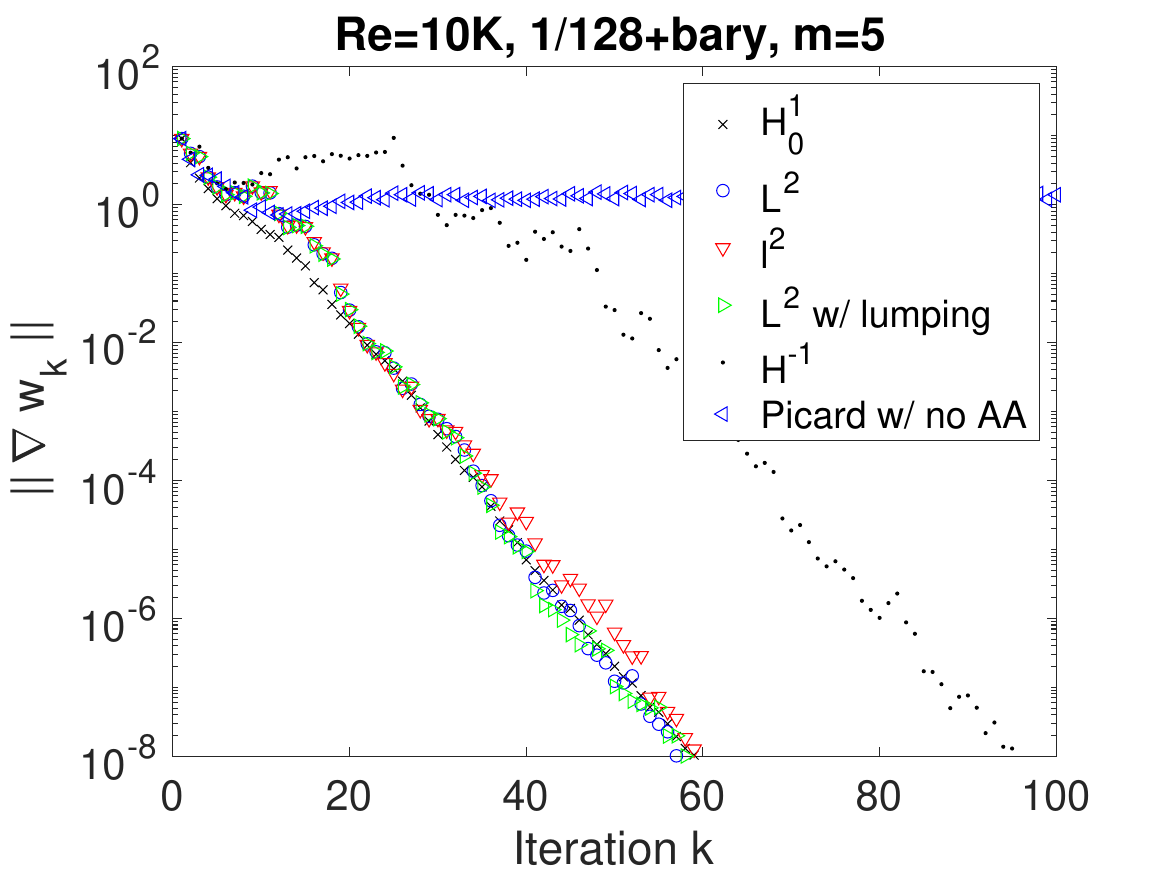}
\includegraphics[trim = 0pt 0pt 0pt 0pt,clip = true,width=.32\textwidth,height=.32\textwidth]{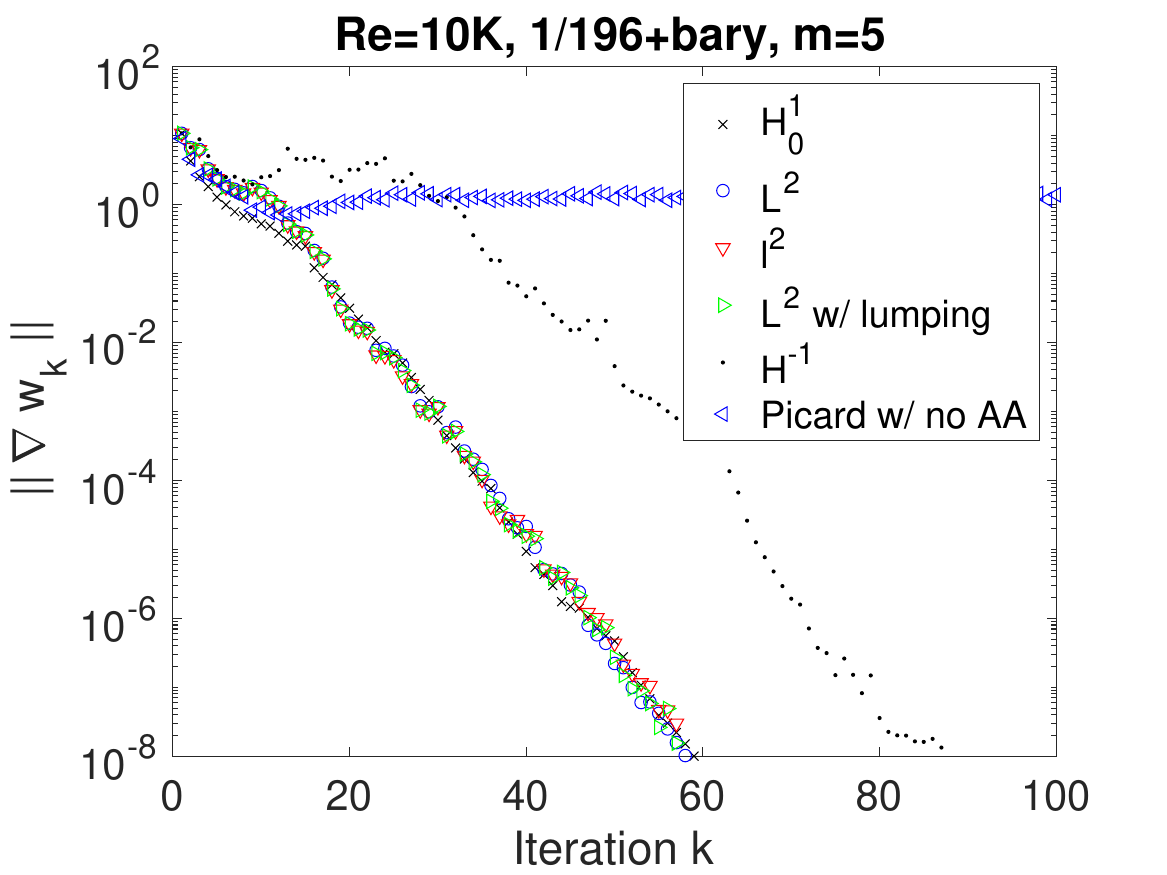}
\caption{\label{dctest1} Shown above are fixed point residuals  for the $Re$=10000 tests for Picard, and with AA-Picard $m=5$ with $H^1_0$, $L^2$, $\ell^2$, diagonally lumped $L^2$, and $H^{-1}$ used in the AA optimization step, on barycenter refined $h=\frac{1}{64},\ \frac{1}{128},\ \frac{1}{196}$ uniform meshes.}
\end{figure}

For our second test problem we use the benchmark 2D driven cavity.  The setup has domain $\Omega=(0,1)^2$, $f=0$, homogeneous Dirichlet boundary conditions enforced on the sides and bottom, and $[ 1,0 ]^T$ on the top (moving lid).  We test with $Re:=\nu^{-1}$=10000, and a plot of the solution is shown in Figure \ref{cavityplots} (and agrees well with solutions in the literature \cite{erturk}).  Our computations use $(P_2,P_1^{disc})$ Scott-Vogelius elements on barycenter refinements of $h=\frac{1}{64},\ \frac{1}{128},\ \frac{1}{196}$ uniform triangular meshes.  For this test, we use optimization norms $H^1_0$, $L^2$, diagonally lumped $L^2$, $\ell^2$, and also $H^{-1}$ with depth $m=5$ AA-Picard.  The latter norm was chosen because the other four give similar results on most of the test problems, and results using the $H^{-1}$ norm will demonstrate that not just any Hilbert space norm can be chosen.  Note that $\| \phi_h \|_{H^{-1}}$ is approximated in our computations by $(\hat \phi_h, S^{-1} \hat \phi_h)^{1/2}$ where $S$ is the stiffness matrix with Dirichlet boundary conditions enforced.

Convergence results are shown in Figure \ref{dctest1}, and show that nearly identical results are found when using any of $H^1_0$, $L^2$, diagonally lumped $L^2$, or $\ell^2$ as the optimization norm.  Results with the $H^{-1}$ norm are clearly worse, needing an extra 30 iterations to converge.  Picard fails to converge in all tests.  We observed no significant differences in convergence behavior for any of the methods on the different meshes.

%

\subsection{3D driven cavity}

\begin{figure}[ht]
\center
\includegraphics[width = .9\textwidth, height=.25\textwidth,viewport=150 0 1250 330, clip]{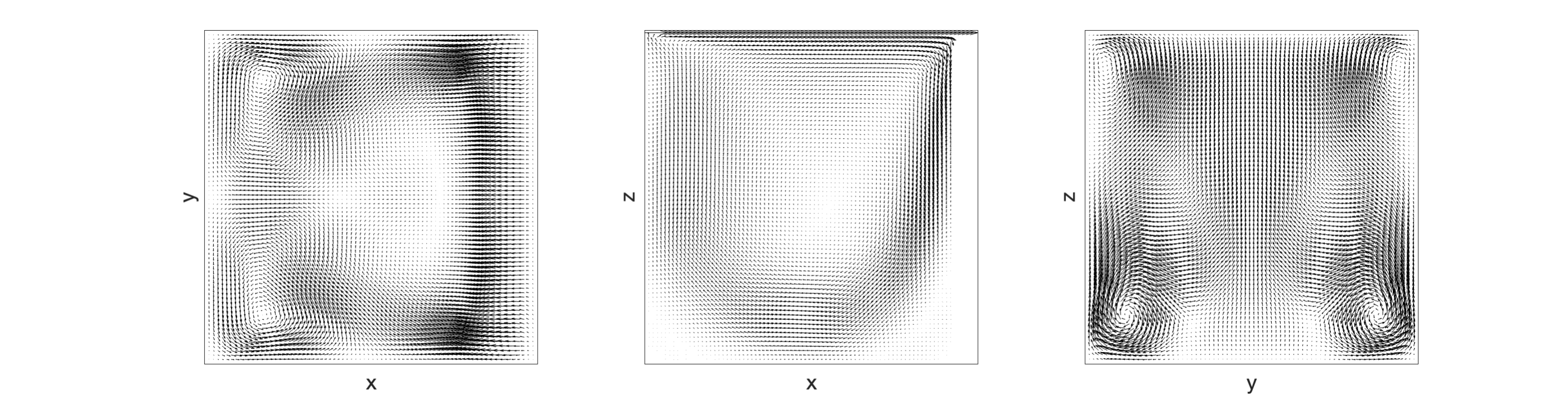}\\
\includegraphics[width = .4\textwidth, height=.3\textwidth,viewport=77 55 450 350, clip]{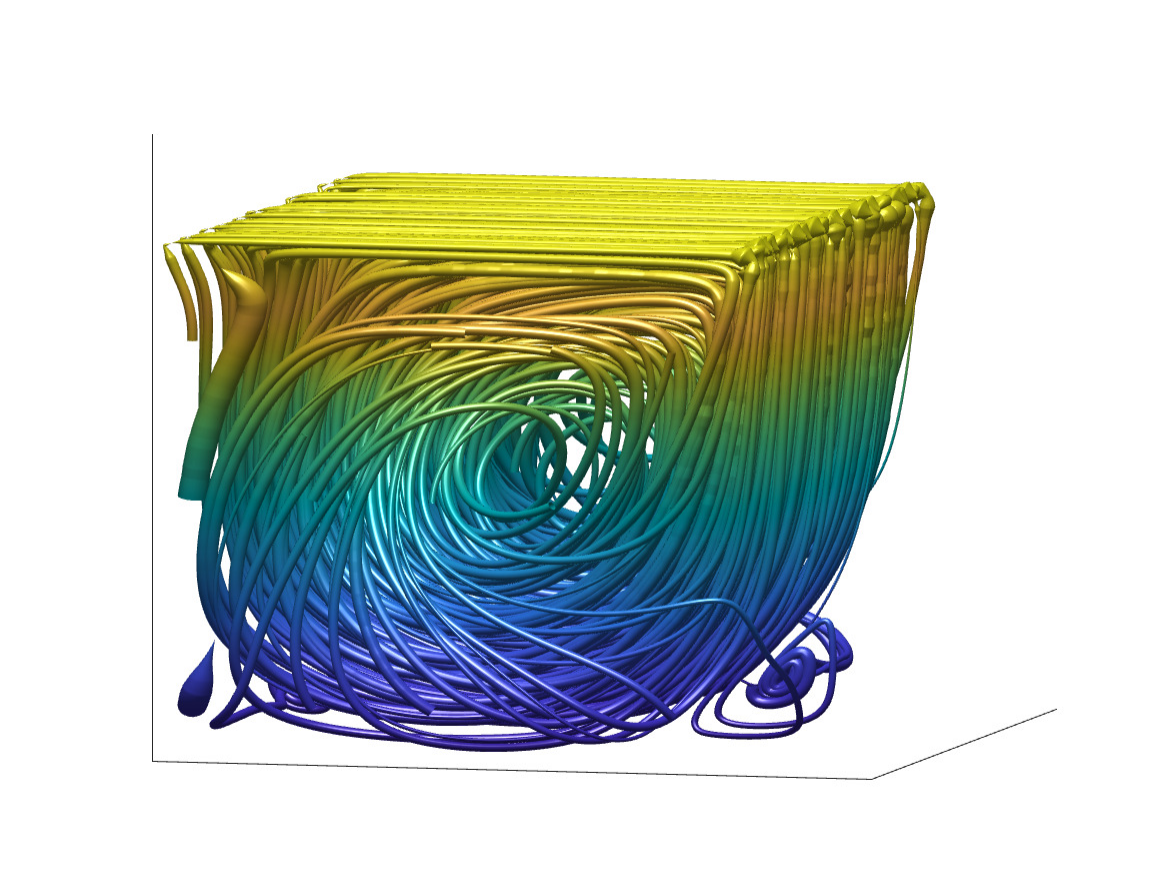}
\caption{\label{fig:midsliceplanes} Shown above are solution plots for the 3D driven cavity problem with $Re$=1000, as velocity vector midsliceplanes and as streamtubes.}
\end{figure}

Our final test is for the 3D lid-driven cavity problem, which is a benchmark 3D analogue of the 2D driven cavity test above.  The domain is the unit cube, there is zero forcing ($f=0$), homogeneous Dirichlet boundary conditions are strongly enforced on the walls and $u=[1,0,0]^T$ is enforced at the top to represent the moving lid.  The viscosity is chosen as the inverse of the Reynolds number, and we will use $Re$=1000 and 1500.  We use $m=20$ in these tests, as previous testing on this problem in \cite{PR25} has shown that larger $m$ performs better.  In this test, we use $H^1_0$, $L^2$, diagonally lumped $L^2$ and $\ell^2$ norms for the AA optimization.

We again compute using three levels of mesh refinement.  The coarsest mesh is constructed by using 8 Chebychev points on [0,1] (and also 0 and 1) to construct a 3D grid of rectangular boxes.  Each box is then split into 6 tetrahedra, and then an additional barycenter refinement splits each of these into 4 tetrahedra.  The mesh is then equipped with $(P_3, P_2^{disc})$ Scott-Vogelius elements, which provides approximately 438K total dof.  It is known from \cite{Z05} that this element choice on such a grid is inf-sup stable.  For the medium mesh, the same process is repeated but using 10 Chebychev points, and this yields 767K total dof.  For the finest mesh, we use the same process as the medium mesh but also refine at the boundary by adding 0.001 and 0.999 to the discretization of [0,1] to create the 3D rectangular boxes.  The finest mesh has 1.3M total dof.  Solution plots found with the finest discretization matched those from the literature \cite{WB02,EPRX20,PR21}, and are shown in Figure \ref{fig:midsliceplanes} for $Re$=1000.

\begin{figure}[ht]
\center
\includegraphics[width = .32\textwidth, height=.32\textwidth,viewport=0 0 550 430, clip]{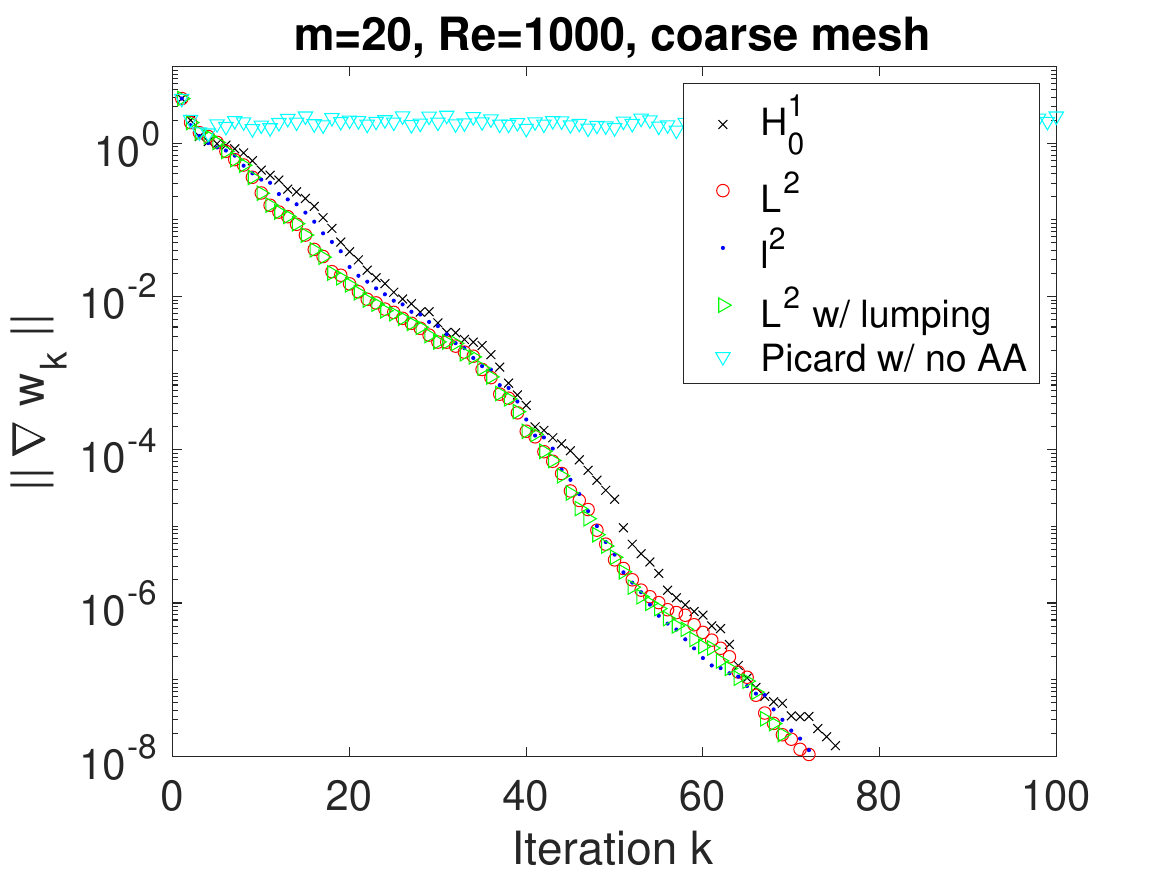}
\includegraphics[width = .32\textwidth, height=.32\textwidth,viewport=0 0 550 430, clip]{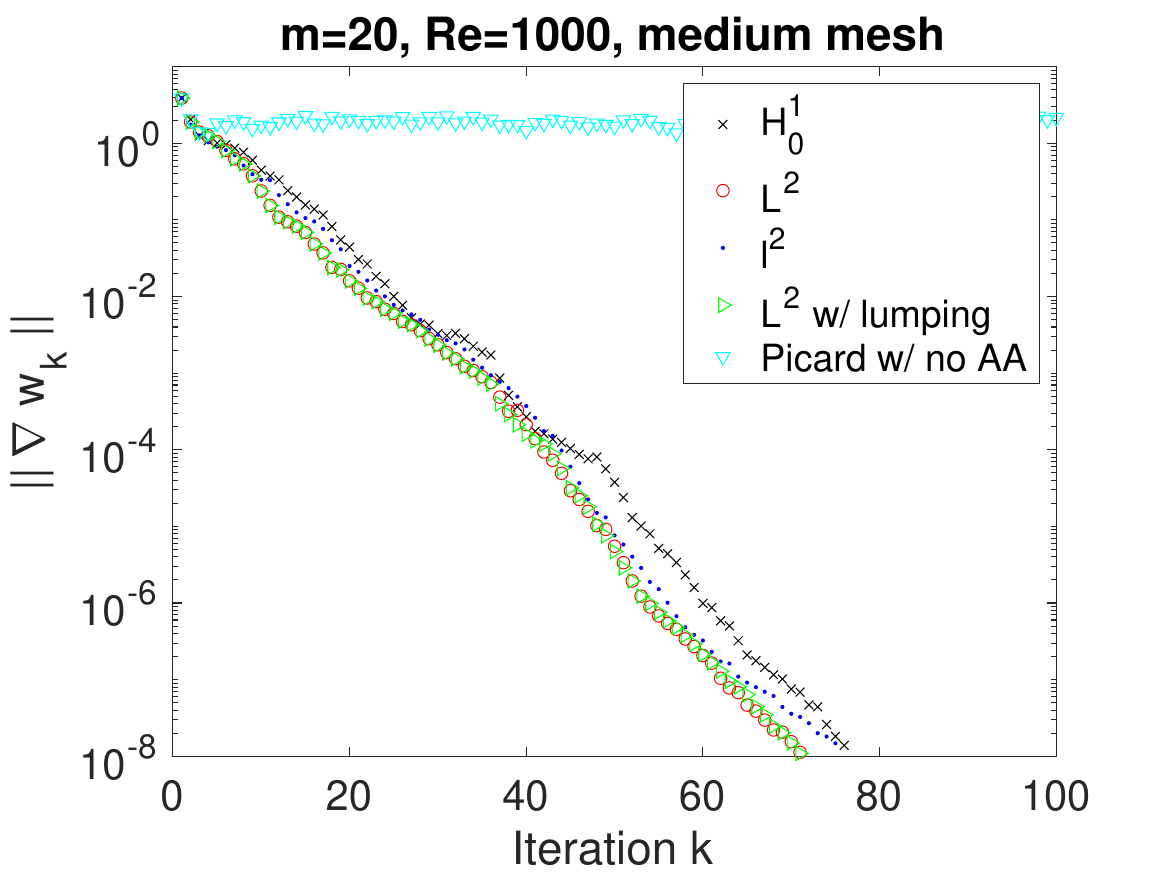}
\includegraphics[width = .32\textwidth, height=.32\textwidth,viewport=0 0 550 430, clip]{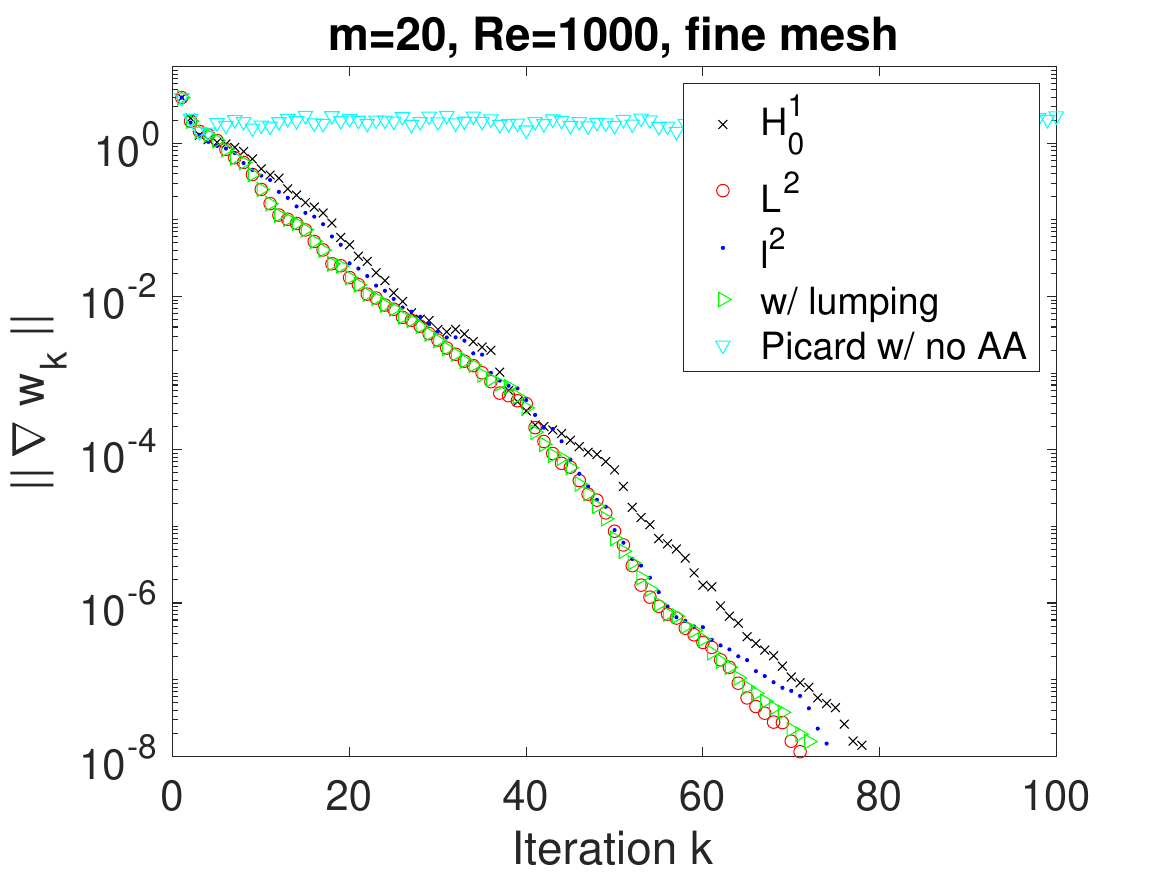}
\caption{\label{cav1000} Shown above are the convergence results for Re=1000 3D driven cavity flow with $m=20$, for varying meshes and varying choices of optimization norm.}
\end{figure}

Convergence results for $Re$=1000 are shown in Figure \ref{cav1000}.  We observe similar results on all three mesh levels, and also similar convergence behavior for each of $H^1_0$, $L^2$, diagonally lumped $L^2$, and $\ell^2$ norm choices for the AA optimization.  On all meshes, the usual Picard iteration without AA failed to converge in 100 iterations.  Convergence results for $Re$=1500 are shown in Figure \ref{cav1500}.  The Picard iteration failed to converge on all three meshes.  For the AA-Picard methods, the overall convergence behavior got slightly worse as the meshes were refined, although not enough to signal a clear mesh dependence.  The $H^1_0$ norm results were the worst on the coarsest and finest meshes but the best on the medium mesh (although in all cases, it was not the best or the worst by much).  Otherwise, the results show that the choices of $H^1_0$, $L^2$, diagonally lumped $L^2$, and $\ell^2$ norms all gave similar results.  

\begin{figure}[ht]
\center
\includegraphics[width = .32\textwidth, height=.32\textwidth,viewport=0 0 550 430, clip]{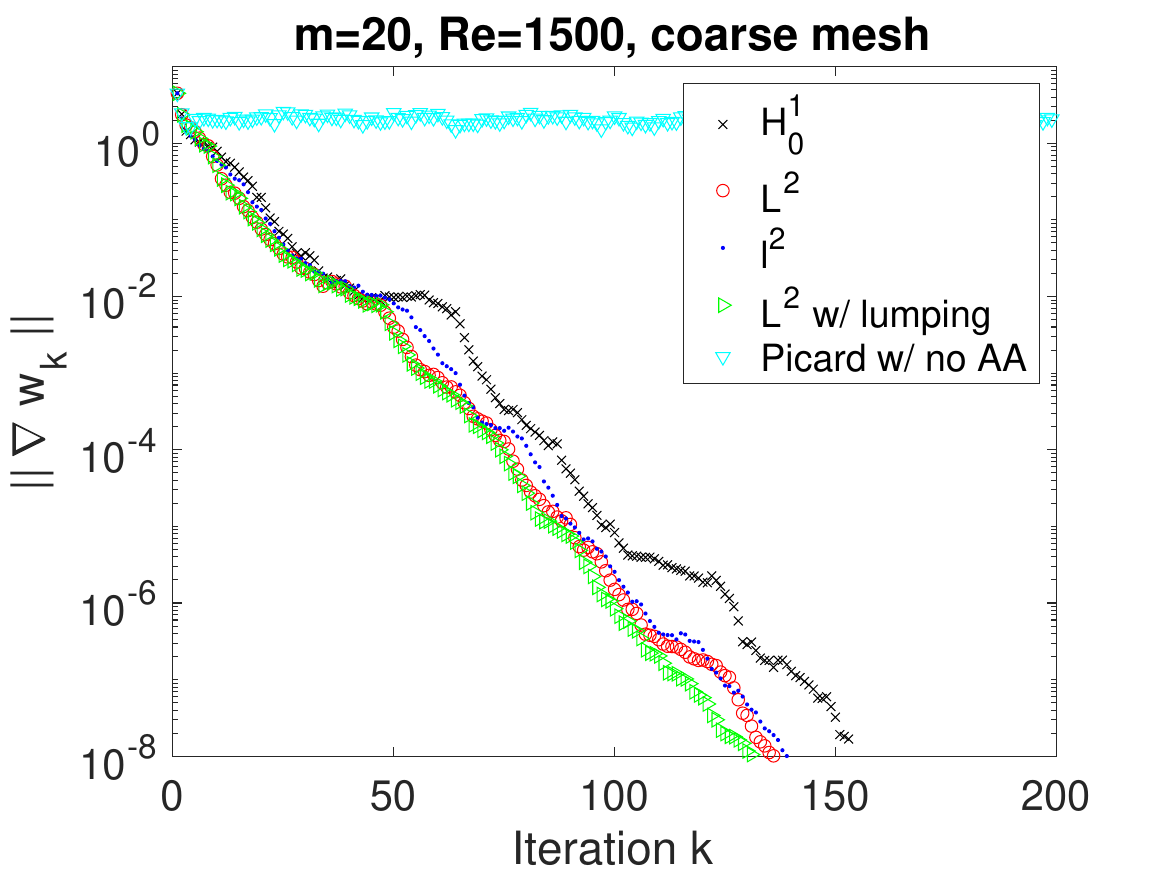}
\includegraphics[width = .32\textwidth, height=.32\textwidth,viewport=0 0 550 430, clip]{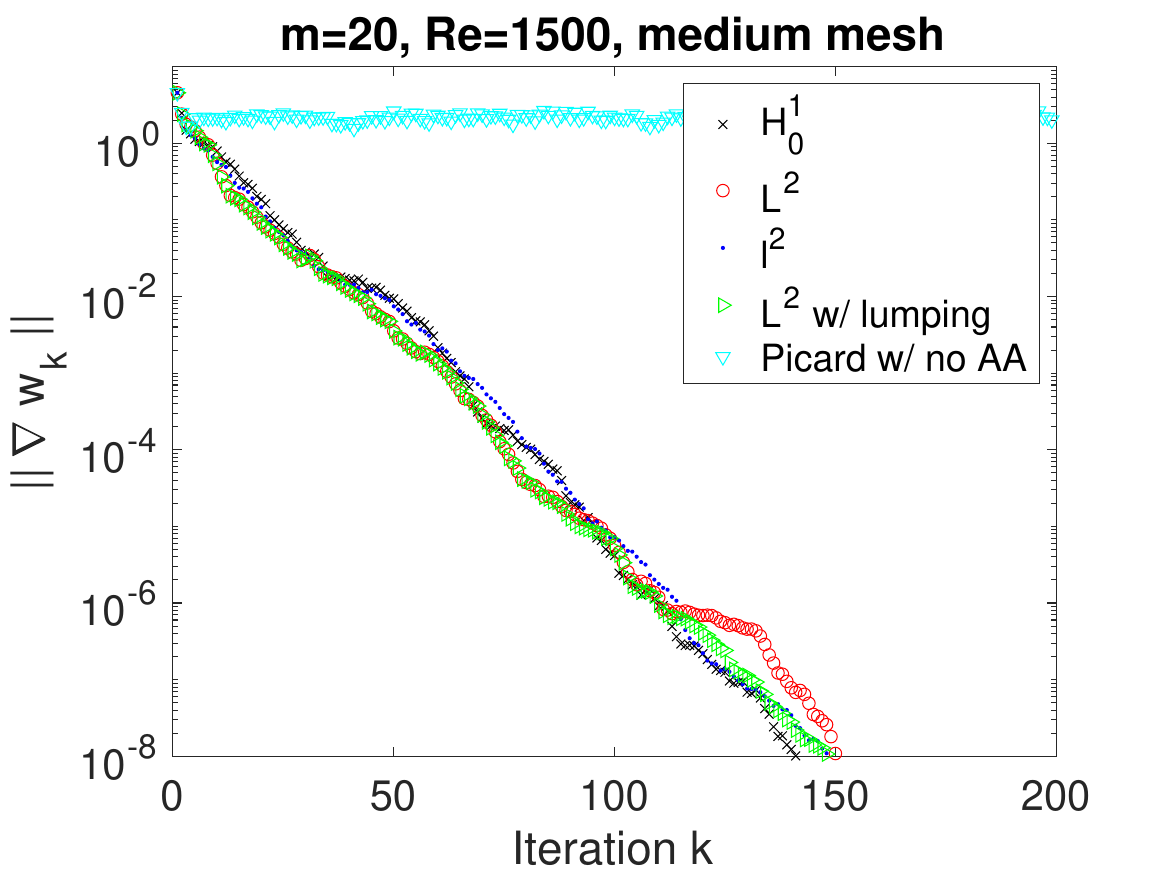}
\includegraphics[width = .32\textwidth, height=.32\textwidth,viewport=0 0 550 430, clip]{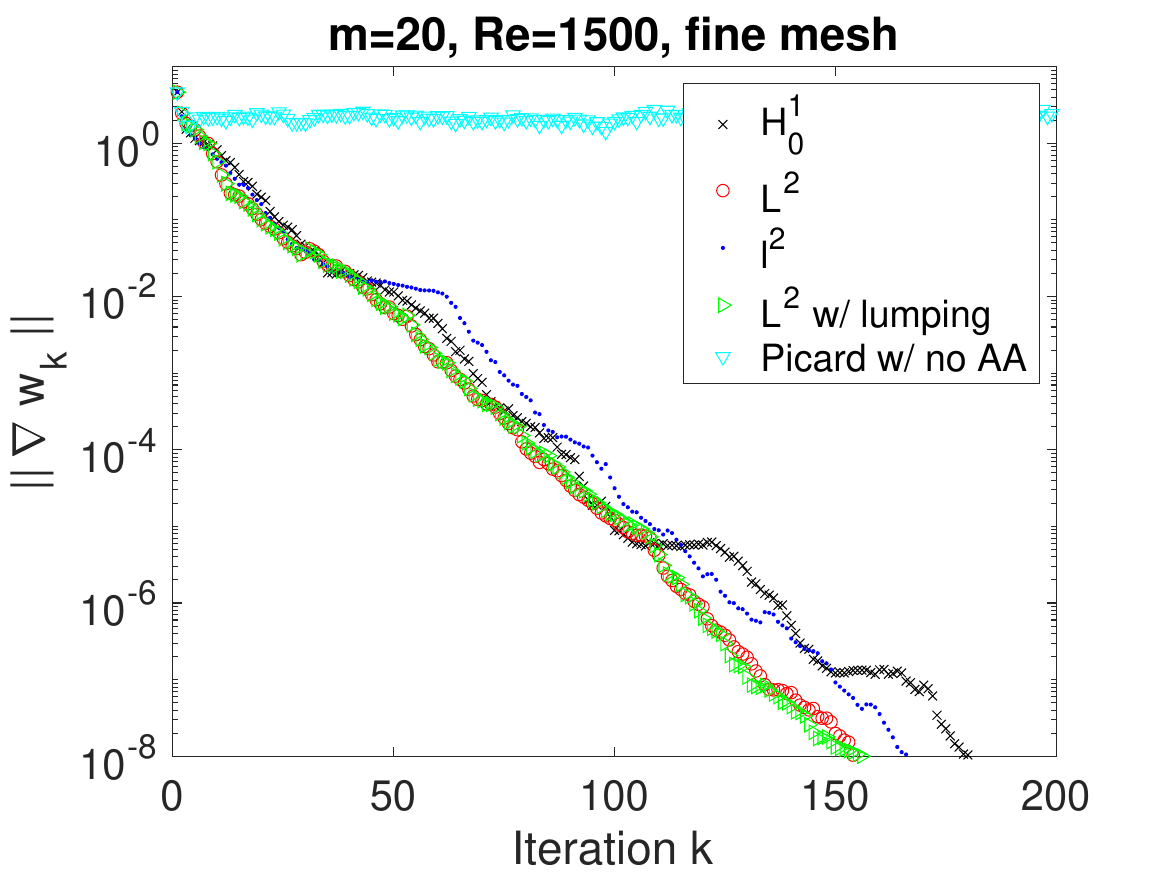}
\caption{\label{cav1500} Shown above are the convergence results for Re=1500 3D driven cavity flow with $m=20$, for varying meshes and varying choices of optimization norm.}
\end{figure}

\section{Conclusions and future directions}

This study investigated the effect of the choice of optimization norm in the AA-Picard iteration.  We have provided new theory that sharpens and generalizes the existing convergence theory of AA-Picard using the $H^1_0$ optimization norm  for the steady NSE by performing a problem-specific analysis instead of invoking the general AA theory from \cite{PR21}.  The improved estimates come from a more accurate upper bounding of the nonlinear terms in the convergence analysis which led to an improved scaling of the first order term in the residual expansion, new AA identities that were incorporated into the nonlinear term expansions, and removal of the assumption of small problem data.  We then proved a new convergence estimate for the same iteration but using the $L^2$ norm for the AA optimization, and found that by writing the first order terms in a particular way the result is observed to be  very similar to that of the $H^1_0$ case.  Our numerical tests showed only one instance where results using these two different norms were not very similar, and this was the coarsest mesh used on channel flow past a cylinder where $L^2$ was somewhat better (on the medium and fine meshes their convergence behavior was nearly identical).

In addition to comparing AA-Picard results using $H^1_0$ and $L^2$ norms, we also numerically tested AA-Picard using diagonally lumped $L^2$ and $\ell^2$.  On the driven cavity tests, even when the meshes were far from uniform, results with these norms gave very similar results to those of $H^1_0$ and $L^2$.  However, significant differences were observed on coarser unstructured meshes for 2D channel flow past a cylinder; here, $\ell^2$ gave significantly worse convergence results while lumped $L^2$ was only worse on the coarsest mesh.

Our numerical tests also provided a grid study for all three experiments, by using three successively refined meshes on each test.  For all three test problems, convergence behavior never got significantly worse as the mesh is refined but in some cases improved as the mesh was refined.  This gives evidence that AA-Picard convergence results will not deteriorate as the grid is refined (provided the grid is sufficiently fine for the $Re$ being used), which is in agreement with our theory that AA-Picard convergence is mesh independent, at least for $H^1_0$ and $L^2$ optimization norm choices.

For future work, a deeper exploration of when $\ell^2$ AA optimization norm performs worse with AA-Picard than $L^2$ should be performed.  While these two norms are related, in our tests $\ell^2$ often performed about the same, exceptwas significantly worse on one test problem (but on the finest mesh for that test problem $\ell^2$ performed the same as $L^2$).  Still, this paper has essentially answered the question {\it `Is it okay to just use the $\ell^2$ norm when using Anderson acceleration?'}; the answer is {\it `Since there is now evidence of $\ell^2$ performing poorly even in the case where the theory and tests show $L^2$ works well, you are better off using the Hilbert space norm of the problem or other norm for which there is a convergence theory.'}  

\section{Acknowledgement}
Author EH was supported in part by NSF DMS 2011490.  Author LR was supported in part by Department of Energy grant DE-SC0025292.

\bibliographystyle{plain}
\bibliography{graddiv}

\end{document}